\documentclass{amsart}

\usepackage{amsmath,amssymb,color}
\usepackage{fullpage}
\usepackage[latin1]{inputenc}
\newtheorem{theorem}{Theorem}[section]
\newtheorem{lemma}[theorem]{Lemma}
\newtheorem{proposition}[theorem]{Proposition}
\newtheorem{corollary}[theorem]{Corollary}
\newtheorem{remark}[theorem]{Remark}

\newcommand{\mc}[1]{{\mathcal #1}}
\newcommand{\mf}[1]{{\mathfrak #1}}
\newcommand{\mb}[1]{{\mathbf #1}}
\newcommand{\bb}[1]{{\mathbb #1}}

\newcommand{\<}{\langle}
\renewcommand{\>}{\rangle}
\newcommand{\x}{\!\otimes}

\begin{document}

\title[Homogenization of generalized second-order elliptic difference operators]{Homogenization of generalized second-order elliptic difference operators}

\author[A.B. Simas]{Alexandre B. Simas}
\address{A.B. Simas
\hfill\break\indent
Departamento de Matemática, Universidade Federal da Paraíba \hfill\break\indent
Cidade Universitária - Campus I, 58051-970, João Pessoa - PB, Brasil}
\email{alexandre@mat.ufpb.br}
\author[F.J. Valentim]{Fábio J. Valentim}
\address{F.J. Valentim
\hfill\break\indent
Departamento de Matemática, Universidade Federal do Espírito Santo\hfill\break\indent
Av. Fernando Ferrari, 514, Goiabeiras, 29075-910, Vitória - ES, Brasil}
\email{fabio.valentim@ufes.br}

\thanks{Research supported by CNPq and FAPES}

\noindent


\begin{abstract}
Fix a function $W(x_1,\ldots,x_d) = \sum_{k=1}^d W_k(x_k)$ where each $W_k: \bb R \to \bb R$ is a strictly increasing right continuous function with left limits.
For a diagonal matrix function $A$, let $\nabla A \nabla_W = \sum_{k=1}^d \partial_{x_k}(a_k\partial_{W_k})$ be a generalized second-order differential operator.
We are interested in studying the homogenization of generalized second-order difference operators, that is, we are interested in the
convergence of the solution of the equation
$$\lambda u_N - \nabla^N A^N \nabla_W^N u_N = f^N$$
to the solution of the equation
$$\lambda u - \nabla A \nabla_W u = f,$$
where the superscript $N$ stands for some sort of discretization. In the continuous case we study the problem in the context of $W$-Sobolev spaces,
whereas in the discrete case  the theory is developed here. The main result is a homogenization result.  Under minor assumptions regarding weak convergence and ellipticity of these matrices $A^N$, we show that  every such sequence admits a homogenization. 
We provide two examples of matrix functions verifying these assumptions:  The first one consists to fix a matrix function $A$ with some minor regularity, and take $A^N$ to be a convenient discretization. The second one consists on the case where $A^N$ represents a random environment associated to an ergodic group, which we then show that the homogenized matrix $A$ does not depend on the realization $\omega$ of the environment. Finally, we apply this result in probability theory. More precisely, we prove a hydrodynamic limit result for some gradient processes.
\end{abstract}

\maketitle

\section{Introduction}
In the '50s William Feller introduced a more general concept of differential operators, that is, operators of the type
$(d/dW)(d/dV)$ where, typically, $W$ and $V$ are strictly increasing functions with $V$ (but not necessarily $W$) being continuous.
In this paper we are interested in the formal adjoint of $(d/dW)(d/dV)$, which is simply $(d/dV)(d/dW)$, in the case $V(x)=x$ is the identity function.
For more details on Feller's operators, we refer the reader to \cite{f1,f2,m}.

Recently, the formal adjoint operator $(d/dx)(d/dW)$ and some non-linear versions were obtained as scaling limits of interacting
particle systems in inhomogeneous media. They may model diffusions with permeable membranes at the points of the discontinuities of $W$, see \cite{f, TC, jl, FSV, v} for further details.
In \cite{v}, for instance, the author introduces an extension of the formal adjoint operator to higher dimensions and provides some results regarding this extension.
Informally,  fix a strictly increasing right continuous functions with left limits and periodic increments, $W_k:\bb R \to \bb R$, $k=1,\ldots,d$, and let $W(x) = \sum_{k=1}^d W_k(x_k)$ for $x\in\bb R^d$. The generalized Laplacian operator is given by $\mc L_W = \sum_{k=1}^d\partial_{x_k}\partial_{W_k}$. Note that $W_k(x) = x$, $k=1,\ldots, d$, corresponds to the classic Laplacian operator. Furthermore, 
In \cite{SV} the authors studied the space of functions $f$ having \textit{weak generalized gradients} $\nabla_W f = (\partial_{W_1} f,\ldots,\partial_{W_d} f)$, which they called the $W$\textit{-Sobolev space.} Several properties, that are analogous to those in classical Sobolev spaces, are also obtained. Recently, in \cite{SV3}, $W$-Sobolev spaces of higher order were defined and results on elliptic regularity were obtained.

Our goal in this paper is to approximate a solution of a certain continuous partial differential equation driven by a generalization of the laplacian
by considering a sequence of discrete functions given by solutions of discrete versions of such equations.

A first contribution of the current paper, presented in Section \ref{sobdis}, is to build a theory in discrete setup that is analogous to the theory
introduced in \cite{SV}. In particular, we introduce the $W$-interpolation and present some results in order to get connections between the discrete and continuous Sobolev spaces.

The main contribution are the results related to homogenization presented in Section \ref{sec5}. Informally speaking, 
we considered discrete problems of the form $\lambda u_N - \nabla^N A^N\nabla_W^Nu_N$ where $A^N$ are sequence of diagonal matrices functions, then we define a  matrix $A$ as being a \textit{homogenization} of the sequence of matrices $A^N$ if some convergences, in appropriate spaces hold, and 
also if the limit is related a continuous problem associated the matrix $A$. 
The main result of this paper, namely Theorem \ref{homoge}, can be summarized in following: under minor assumptions regarding weak convergence and ellipticity on these matrices $A^N$, we show that  every such sequence admits homogenization. 

This result can be incorporated in several situations. We provide two examples that cover many of them. The first consists in fixing a matrix function $A$ with some minor regularity, and let $A^N$ be a convenient discretization. Theorem \ref{homoge} ensures the convergence 
of the solutions $u_N$ to $u_0$, which is associated the problem with $A$. The other example consists of an example of stochastic homogenization. Indeed, let $A^N$ be a sequence of random matrix functions satisfying an ergodicity condition. Then, we show that, this sequence admits homogenization, and that the homogenized matrix
 $A$ does not depend on the realization $\omega$ of the random environment. The focus of this approach is to study the 
asymptotic behavior of the coefficients for a family of random difference schemes, whose coefficients can be obtained by 
the discretization of random high-contrast lattice structures. In this sense, we want to extend the theory of 
homogenization of random difference operators developed in \cite{pr}, which was also tackled in seminal paper by Kozlov \cite{kozlov}.
We also want to generalize its main Theorem (Theorem 2.16) to the context in which we have weak generalized derivatives 
in the sense of $W$-Sobolev spaces.

Finally, as an application of this theory, we prove a hydrodynamic limit result, which is, per se, an interesting result, since it extends
the main result obtained by \cite{v}. More precisely, one must observe that this is a non-trivial extension, since the fact that
the matrix $A$ can be non-constant creates a large difficulty to the method presented there. More precisely, the method consisted in 
studying the spectral decomposition, and the resolvent operator of the operator $\mathcal{L}_W = \sum_{k=1}^d \partial_{x_k}\partial_{W_k}.$
Thus, in our setup, that allows the matrix function $A$ to be non-constant, would force one to look for spectral decomposition and the resolvent
operator of the operator $\sum_{k=1}^d \partial_{x_k}(a_k\partial_{W_k})$.

The remaining of the article is organized as follows: in Section \ref{wsob} we provide a brief review on $W$-Sobolev spaces; in Section \ref{ns}, we define a new space of test functions needed in subsequent sections; in Section \ref{wdiscrete} we introduce a discrete analogous to the continuous $W$-Sobolev spaces, and then, interpolation and projection results connecting the discrete and continuous versions; in Section \ref{sec5} we provide main results of this article; in Section \ref{aplicacao-limite} we apply the results of this article to prove a hydrodynamic limit for gradient processes with conductances in random environments. Finally, we provide an Appendix with a proof to an auxiliary result.

\section{$W$-Sobolev space}\label{wsob}
In this Section we recall some notation and results of  \cite{TC, SV, v}.
Denote by $\bb T^d = ({\bb R}/{\bb Z})^d = [0, 1)^d$ the $d$-dimensional torus and fix a function $W: \bb R^d \to \bb R$ such that

\begin{equation}
\label{w}
W(x_1,\ldots,x_d) = \sum^d_{k=1}W_k(x_k),
\end{equation}
where each $W_k: \bb R \to \bb R$ is a \emph{strictly increasing} right continuous function with left
limits (c\`adl\`ag), with periodic increments, in the sense that for all $u\in \bb R$, $W_k(u+1) - W_k(u) = W_k(1) - W_k(0).$

 Define the generalized derivative $\partial_{W_k}$ of a function $f:\bb T^d \to \bb R$ by
\begin{equation}
\label{f004}
\partial_{W_k} f (x_1,\!\ldots\!,x_k,\ldots, x_d) = \lim_{\epsilon\rightarrow 0}
\frac{f(x_1,\!\ldots\!,x_k +\epsilon,\ldots, x_d)
-f(x_1,\!\ldots\!,x_k,\!\ldots\!, x_d)}{W_k(x_k+\epsilon) -W_k(x_k)}\;,
\end{equation}
if the above limit exists. Denote the generalized gradient of $f$ by
$\nabla_W f = \left(\partial_{W_1}f,\ldots,\partial_{W_d}f\right),$
if the generalized derivatives $\partial_{W_k}$ exist for all $k=1,\ldots,d$.

Let us remember the definition of the space $H_{1,W}(\bb T^d)$, called $W$-Sobolev space.  We denote by $\<\cdot,
\cdot\>$ the inner product of the Hilbert space $L^2(\bb T^d)$ and by
$\Vert \cdot \Vert$ its norm.
Let $L^2_{x^k\x W_k}(\bb T^d)$ be the Hilbert space of
measurable functions $f: \bb T^d\to\bb R$ such that
\begin{equation*}
\int_{\bb T^d}f (x)^2\, d(x^k\x W_k)  \;<\; \infty,
\end{equation*}
where $d(x^k\x W_k)$  represents the product measure in $\bb T^d$
obtained from Lesbegue's measure in $\bb T^{d-1}$ and the measure induced by $W_k$ in $\bb T$:
$$d(x^k\x W_k)\;=\;dx_1\cdots dx_{k-1}\;dW_k\; dx_{k+1}\cdots dx_d.$$
Denote by $\< \cdot, \cdot \>_{x^k\x W_k}$ the inner product of $L^2_{x^k\x W_k}(\bb T^d)$:
\begin{equation*}
\< f,g \>_{x^k\x W_k} \;=\; \int_{\bb T^d}  f(x)
\, g(x)\; d(x^k\x W_k)\, ,
\end{equation*}
 and by $\|\cdot\|_{x^k\x W_k}$ the norm induced by this inner product.

The set $\mc A_{W_k}$ of the eigenvectors of $\frac d{dx}\frac d{dW_k}$ forms a complete orthonormal system in $L^2(\bb T)$, (see \cite{TC}). 
Let
\begin{equation*}
 \mc A_W\;=\;\{f: \bb T^d \rightarrow \bb R;f(x_1,\ldots , x_d)=
\prod^{d}_{k=1}{ f_k(x_k)}, f_k \in  \mc A_{W_k} \},
\end{equation*}
and denote by $\bb D_W:=span(\mc A_W)$. Define the operator $\mc L_W:\bb D_W:=span(\mc A_W) \to L^2(\bb T^d)$ as follows: 
for $f =\prod^d_{k=1}{f_k}\in \mc A_W$,
\begin{equation}
\label{eq31}
\mc L_W(f)(x_1,\ldots x_d)=\sum^{d}_{k =1} \prod^{d}_{j=1, j\neq k}
{f_j(x_j)}\mc L_{W_k}f_k(x_k),
\end{equation}
and extend to $\bb D_W$ by linearity.
In particular,   $\bb D_W$ is dense in $L^2(\bb T^d)$ and,
the set $\mc A_W$ forms a complete, orthonormal, countable system
of eigenvectors for the operator $\mc L_W$. More details can be found in \cite{v}.

Denote by $L^2_{x^k\x W_k, 0}(\bb T^d)$ be the closed subspace of $L^2_{x^k\x W_k}(\bb T^d)$ consisting of the functions $f$ that have zero mean with respect to the measure $d(x^k\x W_k)$:
$$\int_{\bb T^d} f\;  d(x^k\x W_k) = 0.$$

The function $g\in L^2(\bb T^d)$  has $W$-\textit{generalized weak derivative} if for each $k=1,\ldots,d$ there exists a function $G_k\in L^2_{x^k\x W_k,0}(\bb T^d)$ satisfying the following integral by parts identity
\begin{equation}\label{eq22}
\int_{\bb T^d}\big(\partial_{x_k}\partial_{W_K}f\big)\;g\;dx\; =\; -\;\int_{\bb T^d}(\partial_{W_k}f)\;G_kd(x^k\x W_k),
\end{equation}
for every function $f\in \bb D_W$.

Denote by ${H}_{1,W}(\bb T^d)$,  the $W$-Sobolev space,   the set of functions in $L^2(\bb T^d)$ having $W$-generalized weak derivative. Note that $\bb D_W\subset H_{1,W}(\bb T^d)$. Moreover, if $g\in\bb D_W$ then $G_k = \partial_{W_k} g$. For this reason for each function $g\in H_{1,W}$ we denote $G_k$ simply by $\partial_{W_k} g$, and we call it the $k$-th \emph{generalized weak derivative} of the function $g$ with respect to $W$.

In \cite{SV} it is shown that $G_k$ is unique and the set ${H}_{1,W}(\bb T^d)$ is a Hilbert space with respect to the inner product
\begin{equation*}\label{sobolevhilbert}
\<f,g\>_{1,W}\;=\; \<f,g\> + \sum_{k=1}^d\int_{\bb T^d}(\partial_{W_k}f)(\partial_{W_k}g)\;d(x^k\x W_k).
\end{equation*}

Furthermore,  let $H^{-1}_W(\bb T^d)$ be the dual space to $H_{1,W}(\bb T^d)$, that is, $H^{-1}_W(\bb T^d)$ is the set of bounded linear functionals on $H_{1,W}(\bb T^d)$. The following characterization of $H^{-1}_W(\bb T^d)$ can be found in \cite{SV}:
\begin{lemma}
\label{caracterization}
$f\in H^{-1}_{W}(\bb T^d)$ if and only if there exist functions 
$f_0\in L^2(\bb T^d),$ and $f_k\in L^2_{x^k\x W_k,0}(\bb T^d)$, such that
\begin{equation}\label{dual}
f = f_0 - \sum_{k=1}^d\partial_{x_k}f_k,
\end{equation}
in the sense that for $v\in H_{1,W}(\bb T^d)$
$$f(v):=  \int_{\bb T^d} f_0 v dx + \sum_{k=1}^d \int_{\bb T^d} f_k (\partial_{W_k}v) d(x^k\x W_k).$$
Furthermore,
$$\|f\|_{H^{-1}_W} = \inf\left\{\left(\int_{\bb T^d}|f_0|^2 dx + \sum_{k=1}^d\int_{\bb T^d}|f_k|^2 d(x^k\x W_k) \right)^{1/2} ;\quad f\hbox{~satisfies \eqref{dual}}\right\}.$$
\end{lemma}

\section{A  new space of test functions}\label{ns}

We will now define a new space of test functions that will be needed in the proof of the hydrodynamic limit in Section \ref{aplicacao-limite}, and also to prove the Compensated Compactness Theorem (Lemma \ref{compcomp}). 

In this paper we deal with discrete approximations of functions defined on the torus $\mathbb{T}^d$. 
To obtain a hydrodynamic limit, we will eventually need to apply these discretizations to 
test functions of the $W$-Sobolev space. The problem is that these test functions are defined as functions in
$L^2(\mathbb{T}^d)$, and thus, the discretization procedure obtained by restricting the test function to $\bb T_N^d$ is not well-defined. Therefore, we will define a new
space of test functions suitable for discretizations, which we will denote by $\mf D_W(\bb T^d)$.
Note that in \cite{SV} we considered another space for test functions, namely $\bb D_W$, whose definition was given in the previous Section.

Another motivation for such definition is of independent interest: to develop a classical theory of functions
that admit generalized derivatives. In fact, these test functions will be such that one may apply the
operator $\partial_{x_k}\partial_{W_k}$  in the classical sense.

To this end, recall,  from  \eqref{w}, the definition of  $W$.  For $f:\bb T\to \bb R$, let $D(f)$ be the set of its discontinuity points.  For $k=1, \ldots , d$ 
let also $C_{W_k}(\bb T)$ be the set of c\`adl\`ag
functions $f:\bb T \to\bb R$ such that $D(f) \subset D (W_k)$.
We endow $C_{W_k}(\bb T)$ with the sup norm $\|\cdot\|_\infty$.
Equation \eqref{f004} in the one-dimensional case becomes
$$\frac{d f}{dW_k} (x) = \lim_{\epsilon\rightarrow 0} \frac{f(x+\epsilon)
-f(x)}{W_k(x+\epsilon) -W_k(x)},$$
if the above limit exists. 
Let $\mf D_{W_k}$ be the set
of functions $f$ in $C_{W_k} (\bb T)$ such that $\frac{df}{dW_k}(x)$ is
well-defined and differentiable, and that
$\frac{d}{dx}\bigl(\frac{df}{dW_k}\bigr)$
belongs to $C_{W_k}(\bb T)$.

In \cite{TC}, it is proved that $\mf D_{W_k}$ is the set of functions $f$ in $C_{W_k}(\bb T)$
such that
\begin{equation}
\label{f17}
f(x) \;=\; a \;+\; b W_k(x)\; +\; \int_{(0,x]}  dW_k(y) \int_0^y g(z)
\, dz
\end{equation}
for some function $g$ in $C_{W_k}(\bb T)$, where $a$, $b$ are real numbers satisfying
\begin{equation}
\label{f14}
b W_k(1) \;+\; \int_{\bb T} dW_k(y) \int_0^y g(z) \, dz \;=\;0\;, \quad
\int_{\bb T} g(z) \, dz \;=\;0 \;.
\end{equation}
The first requirement corresponds to the boundary condition
$f(1)=f(0)$ and the second one to the boundary condition $df/dW_k (1)
= df/dW_k (0)$.

Let us now define a $d$-dimensional counterpart to the sets $C_{W_k}(\bb T)$ and $\mf{D}_{W_k}$, respectively:
\begin{equation}\label{right-continuous-functions}
C_{W}(\bb T^d) =  {\rm span}\left\{f; f = f_1\otimes \cdots\otimes f_d,\quad f_k\in C_{W_k}(\bb T),k=1,\ldots,d\right\},\,\,\,\,\,\, \text{and,}
\end{equation}
\begin{equation}\label{testfunctions}
\mf D_{W}(\bb T^d) =  {\rm span}\left\{f; f = f_1\otimes \cdots\otimes f_d,\quad f_k\in {\mf D}_{W_k},k=1,\ldots,d\right\},
\end{equation}
where $\otimes$ denotes the tensor product. That is, for functions $f_1\in{\mf D}_{W_1},\ldots, f_d \in{\mf D}_{W_d},$ we have 
$f=f_1\otimes\cdots\otimes f_d$ is such that
$$f(x_1,\ldots,x_d) = \prod_{k=1}^d f_k(x_k),$$
 and ${\rm span}(A)$ stands for the
linear space generated by the set $A$. In other words, $C_W(\bb T^d)$ and ${\mf D}_{W}(\bb T^d)$ are the spaces generated by functions
of the form $f_1\otimes\cdots\otimes f_d$. 

Note that $\mf D_W(\bb T^d)$ is ``almost'' the tensor product of the spaces
$\mf D_{W_1},\ldots,\mf D_{W_d}$, the reason why it is not, is that the tensor product of topological spaces
is defined as a closure of the previous space, and we do not want to take this closure.

\begin{remark}\label{nablaAnablaW}
Note that if $f\in \mf D_W(\bb T^d)$, then $\partial_{W_k} f$ admits a partial derivative in the $k$th direction, therefore, if $a_k\in C_W(\bb T^d)$ is
such that its partial derivative in the $k$th direction exists, then, the function $a_k\partial_{W_k}f$ also admits a partial derivative in the $k$th direction and we have
$$\partial_{x_k}\left(a_k\partial_{W_k}f\right) = (\partial_{W_k}f)(\partial_{x_k}a_k) + a_k \partial_{x_k}\partial_{W_k}f$$
in the strong sense.
\end{remark}

The above remark motivates the following definition. Let $\mathbb{M}_W \subset \left(C_W(\bb T^d)\right)^d$ be the space of functions $a = (a_1,\ldots,a_d)$ such that for every $k=1,\ldots,d$, $a_k$ admits a partial derivative in the $k$th direction. By convenience, we will also say that a diagonal matrix
$A = (a_{kk})$ belongs to $\mathbb{M}_W$ to mean that the function $a = (a_{11},\ldots,a_{dd})$ belongs to $\mathbb{M}_W$.

\begin{remark}\label{suave}
It is easy to see that if $f\in {\mf D}_W(\bb T^d)$, one may apply the operator $\partial_{x_k}\partial_{W_k}$ in the classical sense. Actually, we have that for $f\in {\mf D}_W(\bb T^d)$, and $a\in \mathbb{M}_W$, $\nabla A\nabla_W f\in C_W(\bb T^d)$, where $\nabla A\nabla_W f = \sum_{k=1}^d \partial_{x_k}\left(a_k\partial_{W_k} f\right)$.
\end{remark}

The following result shows that ${\mf D}_W(\bb T^d)$ can be used as a space for test functions:

\begin{proposition}\label{d1}
The space ${\mf D}_W(\bb T^d)$ satisfies
\begin{enumerate}
\item[i)] ${\mf D}_W(\bb T^d)$ is dense in $L^2(\bb T^d)$ in the $L^2$-norm;
\item[ii)] ${\mf D}_W(\bb T^d)$ is dense in $C(\bb T^d),$ the space of continuous functions in $\bb T^d$, in the sup norm $\|\cdot\|_\infty$.
\end{enumerate}
\end{proposition}

\begin{proof}
We begin by noting that the first statement follows directly from the second statement.
We will, thus, prove the second statement. 

Begin by noticing that from Proposition \ref{prova d1} in Appendix, ${\mf D}_{W_k}$ is dense in $C(\bb T)$ in the sup norm. 
To extend this result to the $d$-dimensional case, note that, since $\bb T$ is a compact Hausdorff space, we can use 
Stone-Weierstrass Theorem to conclude that
$$C(\bb T^d) = \bigotimes_{k=1}^d C(\bb T),$$
where $\bigotimes_{k=1}^d C(\bb T)$ is the tensor product of the spaces $C(\bb T)$, that is, $\bigotimes_{k=1}^d C(\bb T)$ is the closure
of the space generated by functions of the form $f_1\otimes\cdots\otimes f_d$, with $f_k\in C(\bb T), k=1,\ldots,d$.

Now, it is easy to see that, since $\mf D_{W_k}$ is dense in $C(\bb T)$, $\mf D_{W}(\bb T^d)$ is dense in 
$C(\bb T^d)$.

\end{proof}

Let us define a function $f:{\bb T}^d \to \mathbb{R}$ as right-continuous if for all $x\in \bb T^d$, we have $\lim_{y\downarrow x} f(y) = f(x)$, where
$y\downarrow x$ means that for all $k=1,\ldots,d$, $y_k\downarrow x_k$.
Observe that for $f_k\in C_{W_k}$ we have $\lim_{y_k\downarrow x_k} f_k(y_k) = f_k(x_k)$. Thus, it is clear that
for $f = f_1\otimes \cdots \otimes f_d$, we have $\lim_{y\downarrow x} f(y) = f(x)$. We then have the following remark:

\begin{remark}
\label{DWrightcontinuous}
Every function $f\in C_W(\bb T^d)$ is right-continuous. In particular, since $\mf D_W(\bb T^d)\subset C_W(\bb T^d)$, every function $f\in \mf D_{W}(\bb T^d)$ is right-continuous.
\end{remark}

\begin{lemma}\label{right-cont-bounded}
Let $u\in C_W(\bb T^d)$, then $u$ is bounded. In particular, if $u\in \mf D_W(\bb T^d)$, then $u$ is bounded.
\end{lemma}
\begin{proof}
From the definition of $C_W(\bb T^d)$, there exist $b_j\in \bb R$, and functions $u_k^j\in C_{W_k}(\bb T)$, $k=1,\ldots, d$, $j=1,\ldots,m$, for some $m\in\bb N$, such that
$$u(x) = \sum_{j=1}^m b_j u_1^j\otimes\cdots\otimes u_d^j (x).$$
Therefore, it is enough to show that each function $u_k^j$ is bounded in $\bb T$. Suppose that $u_k^j$ is not bounded, then, there exists a sequence $(t_n)$ in $\bb T$ such that
$$|u_k^j(t_n)|\stackrel{n\to\infty}{\longrightarrow} \infty\, .$$
By compactness of $\bb T$ and using the identification $\bb T=[0,1)$, we can find a monotone subsequence $t_{n_i}$, such that $\lim_i t_{n_i}\in\bb T$. If $t_{n_i}$ is increasing, then the existence of left limits of $u_k^j$ shows that $|u_k^j(t_{n_i})|$ is bounded, whereas if $t_{n_i}$ is decreasing, the right-continuity shows that $|u_k^j(t_{n_i})|$ is bounded. This contradiction concludes the proof.
\end{proof}

\section{The discrete $W$-Sobolev space}\label{wdiscrete}

In this Section, we introduce the notion of discrete $W$-Sobolev space. As in the continuous case considered in \cite{SV},  we obtain similar results for the discrete setup. Furthermore, we introduce the $W$- interpolation, and we also deal with other types of interpolations. Finally, we present some results in order to get connections between the discrete and continuous Sobolev spaces. 

\subsection{The space $H_{1,W}(\bb T^d_N)$}\label{sobdis}

Denote by  $\bb T^d_N=(\bb Z/N\bb Z)^d = \{0,\ldots,N-1\}^d$ the $d$-dimensional discrete torus with $N^d$ points. 
We will now define and obtain some results on a discrete version of the $W$-Sobolev space.

For $f:\frac 1N\bb T^d_N\to \bb R$ ,
consider the following difference operators: $\partial^N_{x_k}$, which is the standard difference operator on the $k$th canonical direction, and
is given by
\begin{equation}\label{odx}
\partial^N_{x_k}f\left(\frac xN\right)\; =\; N\left[f\left(\frac{x+e_k}{N}\right)-f\left(\frac xN\right)\right]\,\quad\hbox{for}\quad x\in\bb T_N^d;
\end{equation}
and $\partial^N_{W_k}$, which is the $W_k$-difference operator given by
\begin{equation*}%\label{odw}
\partial^N_{W_k}f\left(\frac xN\right)\;=\; \frac{f\left(\frac{x+e_k}{N}\right) -f\left(\frac xN\right)}{W\left(\frac{x+e_k}{N}\right) -W\left(\frac xN\right)},\quad\hbox{for~~}x\in\bb T^d_N.
\end{equation*}

Denote by $L^2(\bb T^d_N)$,  $L^2_{W_k}(\bb T^d_N)$  and $H_{1,W}(\bb T^d_N)$ the Hilbert spaces of the functions defined on $\frac 1N\bb T^d_N$ obtained with respect to the following inner products:
\begin{align*}
 \<f,g\>_N& \;:=\; \frac1{N^d}\sum_{x\in \bb T^d_N} f(x/N)g(x/N),\\
\<f,g\>_{W_k,N} :=& \frac1{N^{d-1}} \sum_{x\in\bb T^d_N}f(x/N)g(x/N)\big(W((x+e_k)/N)-W(x/N)\big),\\
\<f,g\>_{1,W,N}\;:=\; &\<f,g\>_N\; +\; \sum_{k=1}^d \<\partial_{W_k}^Nf,\partial_{W_k}^Ng\>_{W_k,N},
\end{align*}
respectively. Consider the following notation for their induced norms:
$$\| f \|^2_{L^2(\bb T^d_N)} = \<f,f\>_N,\quad \| f \|^2_{L_{W_k}^2(\bb T^d_N)} = \<f,f\>_{W_k,N}\hbox{~and~}\| f \|^2_{H_{1,W}(\bb T^d_N)} = \<f,f\>_{1,W,N},$$
respectively. 

These norms are natural discretizations of the norms considered in continuous case. Using the same arguments used in \cite{SV} to prove Poincar\'e inequality, one can easily prove its discrete version:

\begin{lemma}[Discrete Poincar\'e Inequality]
There exists a finite constant $C$ such that
$$\left\|f-\frac{1}{N^d}\sum_{x\in \bb T^d_N}f (x/N) \right\|_{L^2(\bb T^d_N)} \;\le\;C\|\nabla_W^Nf\|_{L^2_W({\bb T}_N^d)},$$
 for all $f:\frac 1N\bb T_N^d\to\bb R$, where
$$\|\nabla^N_Wf\|_{L^2_W({\bb T}_N^d)}^2 = \sum_{k=1}^d \|\partial_{W_k}^N f\|_{L^2_{W_k}(\bb T_N^d)}^2.$$
\end{lemma}

We will now provide results on discrete elliptic equations. We remark that the proofs of these lemmas are identical to the ones in the continuous case, and the reader is then referred to \cite{SV}. 

Let $\lambda\geq 0$ and $A=(a_{kk}(x))_{d\times d}$, $x\in \bb T^d$, be a diagonal matrix function satisfying the ellipticity condition: there exists a constant $\theta>0$ such that
$\theta^{-1}\le a_{kk}(x) \le \theta,$
for every $x\in \bb T^d$ and every $k=1,\ldots,d$. 
We are interested in studying the problem
\begin{equation}\label{prob TN}
T_{\lambda}^Nu = f,
\end{equation}
where $u:N^{-1}\bb T^d_N \to \bb R$ is the unknown function, $f:N^{-1}\bb T^d_N \to \bb R$ is given, and $T_{\lambda}^N$ denotes the discrete generalized elliptic operator
\begin{equation}\label{def TN}
  T_{\lambda}^Nu\; :=\;\lambda u - \nabla^N A\nabla_W^Nu,
\end{equation}
with
$$\nabla^N A\nabla_W^Nu\;:=\;\sum_{k=1}^d\partial_{x_k}^N\Big(a_{kk}(x/N)\partial_{W_k}^Nu\Big).$$

The bilinear form $B^N[\cdot,\cdot]$ associated with the elliptic operator $T_{\lambda}^N$ is given by
\begin{equation}\label{def BN}
\begin{array}{c}
B^N[u,v] \;=\; \lambda\<u,v\>_N\;+\\*[5pt]
+\frac{1}{N^{d-1}}\sum_{k=1}^d\sum_{x\in\bb T_N^d}a_{kk}(x/N)(\partial_{W_k}^N u)(\partial_{W_k}^Nv)[W_k((x_k+1)/N)-W_k(x_k/N)],
\end{array}
\end{equation}
where $u,v:\frac 1N\bb T_N^d\to\bb R$.

A function $u:\frac 1N\bb T_N^d\to\bb R$ is said to be a weak solution of the equation
$T_\lambda^N u = f\;$ if 
\begin{equation*}
B^N[u,v]\;=\;\<f,v\>_N \;\;\text{for all}\;\;v:N^{-1}\bb T_N^d\to\bb R.
\end{equation*}

Denote by $H^{\bot}_{1,W}(\bb T^d_N)$ the subspace of $H_{1,W}(\bb T^d_N)$ formed by functions $f:\frac 1N\bb T^d_N \to \bb R$ which are orthogonal to the constant functions with respect to the inner product $\<\cdot,\cdot\>_{1,W,N}$. Note that $H^{\bot}_{1,W}(\bb T^d_N)$ is the set of functions $f:\frac 1N\bb T^d_N \to \bb R$ such that
$$\frac1{N^d}\sum_{x\in \bb T^d_N} f(x/N) = 0.$$

\begin{remark}
As seen in Remark \ref{diracdiscreta} below, a function $u$ is a weak solution of the discrete problem \eqref{prob TN} if, and only if, it is a strong solution. We chose to present weak solutions of discrete problems because it is easy to obtain existence and uniqueness results following this approach, and the results are stated in a very similar fashion to those in the continuous case, thus making the analogy with the continuous case more transparent.
\end{remark}

\begin{lemma}
Given a function $f: \frac 1N\bb T^d_N\to \bb R$, the equation
$$\nabla^N A\nabla_W^N u = f,$$
has weak solution $u:\frac 1N\bb T_N^d\to\bb R$ if and only if $f \in H^{\bot}_{1,W}(\bb T^d_N)$. In this case we have uniqueness of the solution disregarding addition by constants. Moreover, if $u\in H^{\bot}_{1,W}(\bb T^d_N)$ we have the bound
$$\| u \|_{{H}_{1,W}(\bb T_N^d)} \leq C \|f\|_{L^2(\bb T_N^d)},$$
where $C>0$ does not depend on $f$ nor on $N$.
\end{lemma}

\begin{lemma}\label{lmdiscreto}
Let $\lambda > 0$ and  $f: \frac 1N\bb T^d_N\to \bb R$. There exists a unique weak solution $u:\frac 1N\bb T_N^d\to\bb R$ of the equation
\begin{equation}\label{eqdisc}
\lambda u - \nabla^N A \nabla_W^N u = f.
\end{equation}
Moreover,
$$\| u \|_{{H}_{1,W}(\bb T_N^d)} \leq C \|f\|_{H_W^{-1}(\bb T_N^d)},\hbox{~and~}\; \| u \|_{L^2(\bb T_N^d)} \leq \lambda^{-1} \|f\|_{H^{-1}_W(\bb T_N^d)},$$
where $C>0$ does not depend neither on $f$ nor on $N$.
\end{lemma}

\begin{remark}\label{diracdiscreta}
Note that in the discrete Sobolev space $H_{1,W}(\bb T_N^d)$ we have a ``Dirac measure'' concentrated in a point $x$ as a function: the function that takes value $N^d$ in $x$ and zero elsewhere. Therefore, we may integrate these weak solutions with respect to this function to obtain that every weak solution is, in fact, a strong solution.
\end{remark}
\begin{remark}
 We had to consider diagonal matrices, because a non-diagonal matrix implies the usage of an operator of the form $\partial_{x_i}\partial_{W_j},$ with $i\neq j$.
 To apply such an operator in the strong sense, the function $W_j$ cannot be discontinuous at any point, and thus this operator loses its main physics motivation, which is to
 create permeable membranes at discontinuity points.
\end{remark}

Motivated by Lemma \ref{caracterization} we are able to obtain a characterization of  the elements of the discrete dual Sobolev space.
\begin{lemma}\label{caracterization discrete}
 Let $f\in H^{-1}_{1,W}(\bb T^d_N)$, the   discrete dual Sobolev space. Then, there exist unique functions $f_k: \frac 1N\bb T^d_N \to \bb R$, $ k=0,1,\ldots, d$ such that
 for all $v:\frac1N\bb T^d_N\to \bb R$, the action of $f$ over $v$ is given by
\begin{align*}
f(v):=  \;\;\; & \<f_0,v\>_N +\sum_{k=1}^d\<f_k, \partial^N_{W_k}v_k\>_{W_k}   \\ 
=\, \frac1{N^d}\sum_{x\in \bb T^d_N} f_0(x/N)v(x/N)+ \frac1{N^{d-1}}&\sum_{k=1}^d\sum_{x\in \bb T^d_N}f_k(x/N)\partial_{W_k}^Nv(x/N)\Big[W((x+ e_k/n) - W(x)\Big].
\end{align*}

Moreover, $\|f\|^2_{H^{-1}_{1,W}(\bb T^d_N)} = \|f_0\|_N^2+\sum_{k=1}^d\|f_k\|_{L^2_{W_k}(\bb T^d_N)}^2$. We denote the functional $f$ by
$$f = f_0 - \sum_{k=1}^d \partial_{x_k}^N f_k.$$
\end{lemma}
\begin{proof}
let $u$ be  the unique weak solution obtained by Lemma \ref{lmdiscreto} when $\lambda = 1$ , $A$ the identidy matrix and, $f\in H^{-1}_{1,W}(\bb T^d_N)$.  In particular, we have $(f,v) = \<u,v\>_{1,W,N}$. Thus, it suffices to consider $f_0 =u$ and $f_k = \partial_{W_k }^N u$, where $k=1,\ldots, d$.
\end{proof}

\subsection{Connections between the discrete and continuous Sobolev spaces}\label{correspondencia}

Let us consider the natural partition of  $\bb T^d$  induced by $\bb T^d_N$. For each $x\in \frac 1N\bb T^d_N$, we denote by $Q_N(x)$ the set
$$Q_N(x)\, =\, \{y\in \bb T^d; \, x_k \leq y_k< x_k +1/N,\,  k=1,2, \ldots, d\, \}\, .$$

Let $u_N: \frac1N\bb T_N^d\to \bb R$ be a mesh function. We will now introduce some interpolation different schemes, that is, procedures to extend $u_N$ from $\frac1N\bb T_N^d$ to the continuous torus $\bb T^d$. The first scheme is the piecewise-constant interpolation given by
$$\widetilde{u}_N(y) = \sum_{x\in\frac 1N \bb T_N^d} u_N(x)\mb 1\{Q_N(x)\}(y).$$ 

However, extend a function $u_N:\frac1N\bb T_N^d\to\bb R$ to $\bb T^d$ in such a way that it belongs to $H_{1,W}(\bb T^d)$ is not straightforward. To do so, we need the definition of $W$-interpolation, which we give below.

Using the standard construction of $d$-dimensional linear interpolation, see for instance \cite{L}, it is possible to define the $W$-interpolation of a function $u_N:\frac1N\bb T_N^d\to \bb R$, with $W(x) = \sum_{k=1}^d W_k(x_k)$ as defined in \eqref{w}. More precisely, let $y\in Q_N(x)$, then the $W$-interpolation of $u_N$, namely, $u_N^\ast$, is given by 
\begin{eqnarray*}
u^\ast_N(y) &=& u_N(x) + \sum_{k=1}^d \partial_{W_k}^N u_N(x)(W_k(y_k)-W(x_k))+\cdots\\
&+&\sum_{k=1}^d \partial_{W_1}^N\cdots\partial_{W_{k-1}}^N\partial_{W_{k+1}}^N\cdots\partial_{W_d}^Nu_N(x) \prod_{\substack{i=1\\i\neq k}}^d(W_i(y_i)-W_i(x_i))\\
&+&  \partial_{W_1}^N\cdots\partial_{W_d}^N u_N(x)\prod_{k=1}^d (W_k(y_k)-W_k(x_k)).
\end{eqnarray*}

Let us consider a new type of interpolation, which are actually $d$ (one for each $m=1,\ldots,d$) interpolations:

\begin{eqnarray*}
u^{(m)}_N(y) &=& u_N(x) + \sum_{\substack{k=1\\k\neq m}}^d \partial_{W_k}^N u_N(x)(W_k(y_k)-W(x_k))+\cdots\\
&+&  \partial_{W_1}^N\cdots\partial_{W_{m-1}}^N\partial_{W_{m+1}}^N\cdots\partial_{W_d}^Nu_N(x) \prod_{\substack{i=1\\i\neq m}}^d (W_i(y_i)-W_i(x_i)).
\end{eqnarray*}

A simple calculation shows that
\begin{equation}\label{deriv-interp}
\frac{\partial u_N^\ast}{\partial W_m}(y) = \left( \partial_{W_m}^Nu_N\right)^{(m)}(y).
\end{equation}

We now establish the connection between the discrete and continuous spaces by showing how a sequence of functions defined in $\frac1N\bb T_N^d$ can converge to a function defined  in $\bb T^d$.

We will say that the sequence of mesh functions $(f_N)_{N\in \bb N}$ in $L^2(\bb T^d_N)$ or
 $L^2_{W_k}(\bb T^d_N)$, $k=1,\ldots, d$, converges strongly (weakly) to the function $f\in L^2(\bb T^d)$ or $L^2_{W_k}(\bb T^d)$, as $N\to \infty$,  if $\widetilde{f_N}\to f$ in $L^2(\bb T^d)$ or $L^2_{W_k}(\bb T^d)$ strongly (weakly), respectively.
Similarly,  a sequence $(u_N)_{N\in \bb N}$ in $H_{1, W}(\bb T^d_N)$ converges strongly (weakly) to $u\in H_{1, W}(\bb T^d)$ if $u_N^* \to u$ strongly (weakly) in $ H_{1, W}(\bb T^d)$.

With respect to convergence of functionals, we say that a sequence of functionals $f_N \in H^{-1}_W(\bb T_N^d)$, with representation
$$f_N = f_{0,N} - \sum_{k=1}^d\partial_{x_k}^N f_{k,N}$$
converge weakly (resp. strongly) to $f\in H^{-1}_W(\bb T^d)$ if the sequence of functionals $F_N$ given by
$$F_N = \widetilde{f}_{0,N} - \sum_{k=1}^d \partial_{x_k}\widetilde{f}_{k,N}$$
converges weakly (resp. strongly) to $f$ in $H^{-1}_W(\bb T^d)$.

\begin{lemma}\label{fraca-interp}
Let $u_N: \frac1N\bb T^d_N\to \mathbb{R}$ be a sequence of functions such that:
\begin{itemize}

\item There exists a constant $C>0$ satisfying   
\begin{equation}\label{cota1}
\|u_N\|_{L²(\bb T_N^d)}^2 =  \frac 1{N^d}\sum_{x\in \bb T^d_n}u_N^2(x/N)\leq C,
\qquad \text{for all}\,\,\, N\geq 1,
\end{equation}
 and, one of the sequences $(u_N^\ast)$, $(\widetilde{u}_N)$ or $(u_N^{(m)})$ (for some $m=1,\ldots,d$) converges weakly in $L^2(\bb T^d)$ to a function $u:\bb T^d\to \bb R$ as $N\to\infty$. Then, the others also converge to $u$ in the same manner.

\item For $j\in \{1,2,\ldots, d\}$, there exists a constant $C_j>0$ satisfying
\begin{equation}\label{cota2}
|u_N\|_{L^2_{W_j}(\bb T_N^d)}^2 =  \frac 1{N^{d-1}}\sum_{x\in \bb T^d_n}u_N^2(x/N)\big[W((x+e_j)/N) - W(x/N)\big]\leq C_j \qquad \text{for all}\,\,\, N\geq 1
\end{equation}
  and, one of the sequences $(u_N^\ast)$, $(\widetilde{u}_N)$ or $(u_N^{(m)})$ (for some $m=1,\ldots,d$) converges weakly in $L^2_{W_j}(\bb T^d)$ to a function $u:\bb T^d\to \bb R$ as $N\to\infty$. Then, the
others also converge to $u$ in the same manner.
\end{itemize}
\end{lemma}

\begin{proof}

We start with the $L^2(\bb T^d)$ case. Suppose $(u_N^\ast)$ converges weakly to $u$ in $L^2(\bb T^d)$. We will prove that $(\widetilde{u_N})$ also converges to $u$. The remaining cases can be handled in a similar manner. Note that \eqref{cota1} implies the uniform boundedness of the norms $\|u_N^\ast\|_{L^2(\bb T^d)}$, $\|\widetilde{u}_N\|_{L^2(\bb T^d)}$ and $\|u_N^{(m)}\|_{L^2(\bb T^d)}$ because the values of the functions $u_N^\ast$, $\widetilde{u}_N$ and $u_N^{(m)}$ in each cell $Q_N(x)$, $x\in \bb T^d_N$, lies between the largest and smallest values of $u_N$ at  the vertices of $Q_N(x)$. Hence, each of theses sequences is weakly compact in $L^2(\bb T^d)$. From this, to prove the convergence we are interested, it is enough to prove the convergence $\int_{\bb T^d}\widetilde{u}_N\Phi\; dy \to \int_{\bb T^d}u\Phi\; dy $ for all $\Phi \in C^\infty( \bb T^d)$ (notice that we have, by assumption, that $\int_{\bb T^d}u_N^\ast\Phi\; dy \to \int_{\bb T^d}u\Phi\; dy$).

Let $\widetilde{\Phi}_N$ be the piecewise-constant function  which coincides with $\Phi$ at the lattice-points  $\frac 1N\bb T^d_N$. By uniform convergence of $\widetilde{\Phi}_N$ to $\Phi$, we have $\int_{\bb T^d}u_N^\ast\widetilde{\Phi}_N\; dy \to \int_{\bb T^d}u\Phi \; dy$.  Let us consider
\begin{equation*}
R_N:= \int_{\bb T^d} \big(u_N^\ast - \widetilde{u}_N\big)\widetilde{\Phi}_N\,dy = 
\sum_{x\in \frac1N\bb T^d_N}\Phi(x)\int_{Q_N(x)}\big(u_N^\ast - \widetilde{u}_N\big)\, dy.
\end{equation*}
By explicit form of the interpolations, we have
\begin{eqnarray*}
R_N &=& \sum_{x\in \frac1N\bb T^d_N}\Phi(x)\Big\{  \sum_{k=1}^d \partial_{W_k}^N u_N(x)\int_{Q_N(x)}(W_k(y_k)-W(x_k))\, dy+\cdots\\
&+&\sum_{k=1}^d \partial_{W_1}^N\cdots\partial_{W_{k-1}}^N\partial_{W_{k+1}}^N
\cdots\partial_{W_d}^Nu_N(x)\int_{Q_N(x)} \prod_{\substack{i=1\\i\neq k}}^d(W_i(y_i)-W_i(x_i))\, dy\\
&+&  \partial_{W_1}^N\cdots\partial_{W_d}^N u_N(x)\int_{Q_N(x)}\prod_{i=1}^d (W_i(y_i)-W_i(x_i))\, dy   \Big\}\, .
\end{eqnarray*}

After a summation by parts we transfer in each term inside the braces one of the $W$-differences from $u_N$ to $\Phi$ and thus, the resulting expression is
 \begin{align*}
R_N = -\; \sum_{x\in \frac1N\bb T^d_N}\Big\{  \sum_{k=1}^d u_N(x)&\partial_{\bar{W}_k}^N\Phi (x)\, \int_{Q_N(x)}(W_k(y_k)-W(x_k))\, dy+\cdots\\
&+  \partial_{\bar{W}_1}^N\Phi (x)\, \partial_{W_2}^N\ \cdots\partial_{W_d}^N u_N(x)\prod_{i=1}^d\int_{x_i}^{x_i+ \frac 1N} (W_i(y_i)-W_i(x_i))\, dy_i   \Big\}\, ,
\end{align*}
where
\begin{equation*}%\label{odw}
\partial^N_{\bar{W}_k}f\left(\frac xN\right)\;=\; \frac{f\left(\frac xN\right) - f\left(\frac{x-e_k}{N}\right) }{W\left(\frac{x+e_k}{N}\right) -W\left(\frac xN\right)},\quad\hbox{for~~}x\in\bb T^d_N,
\end{equation*}
and  the exchange of the order of integration is due to Fubini's Theorem. Note that
\begin{align*}
\Big( \partial_{\bar{W}_{i_1}}^N\Phi (x)&\partial_{W_{i_2}}^N\ \cdots\partial_{W_{i_s}}^N u_N(x)\prod_{j=1}^s\int_{x_{i_j}}^{x_{i_j}+1/N} [W_{i_j}(y_{i_j})-W_{i_j}(x_{i_j})]\, dy_i\Big)^2\, \leq \\
\Big(&\frac{1}{N^d}\partial_{\bar{W}_{i_1}}^N\Phi (x)\partial_{W_{i_2}}^N\ \cdots\partial_{W_{i_s}}^N u_N(x)\prod_{j=1}^s[W_{i_j}(x_{i_j}+1/N)-W_{i_j}(x_{i_j})]\Big)^2\; = \\
\Big(\frac{1}{N^d}\partial_{\bar{W}_{i_1}}^N\Phi &(x)[W_{i_1}(x_{i_1}+1/N)-W_{i_1}(x_{i_1})]    \Big)^2\Big(
\partial_{W_{i_2}}^N\ \cdots\partial_{W_{i_s}}^N u_N(x)\prod_{j=2}^s[W_{i_j}(x_{i_j}+1/N)-W_{i_j}(x_{i_j})]\Big)^2 = \\
&\Big(\frac{1}{N^d}\partial_{\bar{x}_{i_1}}^N\Phi (x)\frac{1}{N}    \Big)^2\Big( C\sum_{i=1}^du_N^2(x+ e_i/N)\Big)^2 
\end{align*}
where in the last equality the constant $C>0$ does not depend on $N$. 

In this way, $|R_N|$ is bounded by a finite sum (that depends on $d$)  of terms of the form
\begin{align*}
\frac{C}{N}\frac{1}{N^d}\sum_{x\in \frac1N\bb T^d_N} |\partial_{\bar{x}_{i_1}}^N\Phi (x)|  \big( \sum_{i=1}^du_N^2(x+& e_i/N)\big)^{1/2} \leq 
\frac{C}{N}\frac{1}{N^d}\Big(\sum_{x\in \frac1N\bb T^d_N} \partial_{\bar{x}_{i_1}}^N\Phi^2 (x)\Big)^{1/2} \Big(\sum_{x\in \frac1N\bb T^d_N}  \sum_{i=1}^du_N^2(x+ e_i/N)\Big)^{1/2} \\
\leq\, \frac{Cd}{N}\Big(\frac{1}{N^d}&\sum_{x\in \frac1N\bb T^d_N}  \partial_{\bar{x}_{i_1}}^N\Phi^2 (x)\Big)^{1/2} \Big(\frac{1}{N^d}\sum_{x\in \frac1N\bb T^d_N}  u_N^2(x)\Big)^{1/2},
\end{align*}
where the previous estimates  follow from Cauchy-Schwarz inequality. By hypothesis \eqref{cota1} and the fact that $\Phi \in C^\infty(\bb T^d)$ the sums inside the brackets are bounded. Thus, $R_N$ goes to zero as $N\to \infty$. Therefore, we have shown that
$$\int_{\bb T^d}\widetilde{u}_N\widetilde{\Phi}_N\; dx\to \int_{\bb T^d}u\Phi\; dx.$$

Hence
$$ \int_{\bb T^d}\widetilde{u}_N\Phi\; dx = 
\int_{\bb T^d}\widetilde{u}_N\widetilde{\Phi}_N\; dx  + 
\int_{\bb T^d}\widetilde{u}_N[\Phi - \widetilde{\Phi}_N]\; dx\to \int_{\bb T^d}u\Phi\; dx.$$

This concludes the proof for the first case. 
Now, let us consider the  $L^2_{W_j}(\bb T^d)$ case.  
The proof is similar to previous case. To keep notation simple, fix, without loss of generality, $j=1$. The beginning of the proof is exactly the same as in the previous case, one just needs to replace the Lebesgue measure $dy$ by the product measure $d(y_1\otimes W_1)$. The essential difference comes in the estimation of the term 

\begin{equation*}
\Big( \partial_{\bar{W}_{i_1}}^N\Phi (x)\partial_{W_{i_2}}^N\ \cdots\partial_{W_{i_s}}^N u_N(x)\int_{Q_N(x)}\prod_{j=1}^s [W_{i_j}(y_{i_j})-W_{i_j}(x_{i_j})]\, d(y_1\otimes W_1)\Big)^2.
\end{equation*}
We will now provide this estimation. The previous term is bounded above by
\begin{align*}
 \Big(\frac{1}{N^{d-1}}\partial_{\bar{W}_{i_1}}^N\Phi (x)\partial_{W_{i_2}}^N\ \cdots\partial_{W_{i_s}}^N& u_N(x)[W_1(x_{1}+1/N)-W_{1}(x_1)]\prod_{j=1}^s[W_{i_j}(x_{i_j}+1/N)-W_{i_j}(x_{i_j})]\Big)^2\;  \\
=\Big(\frac{[W_1(x_{1}+\frac 1N)-W_{1}(x_1)]}{N^{d-1}}&\partial_{\bar{W}_{i_1}}^N\Phi (x)[W_{i_1}(x_{i_1}+\frac 1N)-W_{i_1}(x_{i_1})]    \Big)^2  \, \times \\
 \times  \, \Big(\partial_{W_{i_2}}^N\,&  \cdots \partial_{W_{i_s}}^N u_N(x)\prod_{j=2}^s[W_{i_j}(x_{i_j}+\frac 1N)-W_{i_j}(x_{i_j})]\Big)^2  \\
\leq\Big(\frac{[W_1(x_{1}+\frac 1N)-W_{1}(x_1)]}{N^{d-1}}&\partial_{\bar{x}_{i_1}}^N\Phi (x)\frac{1}{N}    \Big)^2\Big( C_1\sum_{i=1}^du_N^2(x+ e_i/N)\Big)^2 \, .
\end{align*}

So the new term $|R_N|$ is, again, bounded by a finite sum (that depends on $d$)  of terms of the form
\begin{equation*}
\frac{C_1}{N}\Big(\frac 1{N^{d-1}}\sum_{x\in \frac1N\bb T^d_N}[W_1(x_{1}+\frac 1N)-W_{1}(x_1)] \partial_{\bar{x}_{i_1}}^N\Phi^2 (x)\Big)^{1/2} \Big(\frac{1}{N^{d-1}}\sum_{x\in \frac1N\bb T^d_N} [W_1(x_{1}+\frac 1N)-W_{1}(x_1)] u_N^2(x)\Big)^{1/2}\, .
\end{equation*}
By hypothesis \eqref{cota2} and the fact that $\Phi \in C^\infty(\bb T^d)$ the sums inside the brackets are bounded and thus $R_N$ goes to zero as $N\to \infty$. Therefore, we have shown that
$$\int_{\bb T^d}\widetilde{u}_N\widetilde{\Phi}_N\; d(y_1\otimes W_1)\to \int_{\bb T^d}u\Phi\; d(y_1\otimes W_1).$$

Hence
$$ \int_{\bb T^d}\widetilde{u}_N\Phi\; d(y_1\otimes W_1) = 
\int_{\bb T^d}\widetilde{u}_N\widetilde{\Phi}_N\; d(y_1\otimes W_1)  + 
\int_{\bb T^d}\widetilde{u}_N[\Phi - \widetilde{\Phi}_N]\; d(y_1\otimes W_1)\to \int_{\bb T^d}u\Phi\; d(y_1\otimes W_1).$$
This concludes the proof.

\end{proof}
We will now state and prove the strong counterpart to the above Lemma:

\begin{lemma}\label{strong-interp}
Let $u_N:\mathbb{\bb T}^d\to \mathbb{R}$ be a sequence of functions and $C> 0$ be a constant such that
\begin{equation}\label{D}
\|u_N\|_{H_{1,W}(\bb T_N^d)}\leq C \qquad \text{for all}\,\,\, N\geq 1\, .
\end{equation}
Then, if one of the sequences $(u_N^\ast)$, $(\widetilde{u}_N)$ or $(u_N^{(m)})$ (for some $m=1,\ldots,d$) converges in $L^2(\bb T^d)$ to a function $u:\bb T^d\to \bb R$ as $N\to\infty$, then the
others also converge to $u$ in the same manner.

\end{lemma}
\begin{proof}
Let us consider the case $u_N^* \to u$ in $L^2(\bb T^d)$ and we will prove that $u^{(m)}_N\to u$ 
in $L^2(\bb T^d)$ .  The remaining cases can be proved analogously. Recall the notation introduced at the beginning of this subsection and 
note that 
\begin{equation}\label{RN}
R_N = \int_{\bb T^d} (u_N^* - u^{(m)}_N)^2\,dy \leq \sum_{x\in \frac1N\bb T^d_N} \int_{Q_N(x)} (u_N^* - u^{(m)}_N)^2\,dy \, .
\end{equation}

A simple computation show that the quantity $|u_N^* (y)- u^{(m)}_N(y)|$ assumes its largest value on $\bar{Q}(x)$ at one of the vertices of $Q_N(x)$ and that this value is equal to 
$\big(W_m(x_m +1/N) - W_m(x_m)\big)|\partial_{W_m}^{(m)}u_N |$. That is,
\begin{equation}
max_{y\in \bar Q_N(x)} |u_N^* (y)- u^{(m)}_N(y)| = \big(W_m(x_m +1/N) - W_m(x_m)\big)\,max_{y;\, |y-x|=1/N,\; y_m = x_m}|\partial_{W_m}^{N}u_N(y)|.
\end{equation}

In fact, 
\begin{align*}
u^\ast_N(y) - u_N^{(m)}(y) =&     \Big(W_m(y_m) - W_m(x_m)\Big)\Big\{\partial_{W_m}^{N} u_N(x) + \sum_{j=1; j\neq m}^d \partial_{W_j}^N \partial_{W_m}^{N}u_N(x)\big(W_j(y_j)-W(x_j)\big)+\cdots\\
+&\,\sum_{j=1; j\neq m}^d \partial_{W_1}^N\cdots\partial_{W_{j-1}}^N\partial_{W_{j+1}}^N\cdots\partial_{W_d}^N\partial_{W_m}^{N}u_N(x) \prod_{\substack{i=1\\i\neq j}}^d\big(W_i(y_i)-W_i(x_i)\big)\\
+& \,  \partial_{W_1}^N\cdots\partial_{W_d}^N \partial_{W_m}^{N}u_N(x)\prod_{i=1;\, i\neq m}^d \big(W_i(y_i)-W_i(x_i)\big)\Big\}\,\, =\\
 \Big(&W_m(y_m) - W_m(x_m)\Big) (\partial_{W_m}^{N}u_N)^{\ast '}(y_1,\cdots , y_{m-1},y_{m+1},\cdots, y_d)\, 
\end{align*}
where $(\partial_{W_m}^{N}u_N)^{\ast '}$ denotes the $W-$ interpolation of the function $\partial_{W_m}^{N}u_N$  in the $(d-1)$-dimensional setup.

In this way, $R_N$ is   bound above  by
\begin{align*}
\sum_{x\in \frac1N\bb T^d_N}  \int_{\bar Q_N(x)} \Big(\big[W_m(x_m +1/N) - W_m(x_m)\big]max_{z;\, |z-x|=1/N,\; z_m = x_m}|\partial_{W_m}^{N}u_N(z)|\Big)^2\, dy = \\
\sum_{x\in \frac1N\bb T^d_N}  \frac 1{N^d}\Big(\big[W_m(x_m +1/N) - W_m(x_m)\big]max_{z;\, |z-x|=1/N,\; z_m = x_m}|\partial_{W_m}^{N}u_N(z)|\Big)^2\leq \\
C_1\frac{ W_m(1)-W_m(0)}{N}\sum_{z\in \bb T^d_N}\Big(\partial_{W_m}^{N}u_N(z)\Big)^2 \big[W_m(z_m +1/N) - W_m(z)\big]\frac 1{N^{d-1}} \; = \\
C_1\frac{ W_m(1)-W_m(0)}{N}\|\partial_{W_m}^Nu_N\|^2_{W_m, N}\, \leq\, 
C_1\frac{ W_m(1)-W_m(0)}{N}\|u_N\|^2_{H_{1, W (\bb T^d_N)}}
\end{align*}
where in the previous expression, $C_1$ is a constant that depends on $d$ (one may take, for instance, $C_1=2^d$). So, by \eqref{D},  $\lim_{N\to \infty} R_N = 0$.

Thus, it follows that
$$\int_{\bb T^d} (u^{(m)}_ N - u)^2\, dy \leq 2 \int_{\bb T^d} ( u^*_N - u^{(m)}_ N)^2 +( u^*_N - u)^2  \, dy \to 0,$$
as $N\to \infty$. This concludes the proof.

\end{proof}
Finally, we have the following Lemma regarding strong and weak compactness:

\begin{lemma}\label{compact-interp}
Let $u_N:\mathbb{\bb T}^d\to \mathbb{R}$ be a sequence of functions such that there exists some constant $C\geq 0$
such that, for every $N\geq 1$,
\begin{equation}\label{L}
\|u_N\|_{H_{1,W}(\bb T_N^d)}\leq C.
\end{equation}
Then, the sequence $(u_N^\ast)$ forms a uniformly bounded set in $H_{1,W}(\bb T^d)$, which is therefore strongly precompact in $L^2(\bb T^d)$ and weakly precompact in $H_{1,W}(\bb T^d)$.
\end{lemma}
\begin{proof}

To prove the lemma, it is enough to show that \eqref{L} implies the uniform  boundedness of $\|u^*_N\|_{H_{1,W}(\bb T^d)}$ and then we apply \cite[Proposition 2.9]{ SV}.

Note that
\begin{align*}
\|u^*_N\|_{L^2(\bb T^d)}^2 \, = \,    \int_{\bb T^d} (u^*_N(y))^2\; dy  =& \sum_{x\in \frac1N\bb T^d_N}  \int_{Q_N(x)} (u^*_N(y))^2\; dy \leq \\
\sum_{x\in \frac1N\bb T^d_N}  \int_{Q_N(x)} \sum_{z \in \frac 1N \bb T^d_N;\,\, |z-x|=1/N}   (u_N(z))^2\; dy\; 
\leq\; &\frac{2^d}{N^d}\sum_{x\in \bb T^d_N}u_N^2(x) = 2^d\|u_N\|_{L^2(\bb T^d_N)}^2\;, 
\end{align*}
where $z$, in the previous expression, represents the summation over all vertices of the cell $Q_N(x)$. Further, for  $m=1,2, \ldots, d,$ we have, by an explicit computation,
\begin{align*}
\partial_{W_m} u^*_N(y)\,\, =&\,\,
\partial_{W_m}^{N} u_N(x) + \sum_{j=1; j\neq m}^d \partial_{W_j}^N \partial_{W_m}^{N}u_N(x)\big(W_j(y_j)-W(x_j)\big)+\cdots\\
+&\,\sum_{j=1; j\neq m}^d \partial_{W_1}^N\cdots\partial_{W_{j-1}}^N\partial_{W_{j+1}}^N\cdots\partial_{W_d}^N\partial_{W_m}^{N}u_N(x) \prod_{\substack{i=1\\i\neq j}}^d\big(W_i(y_i)-W_i(x_i)\big)\\
&+ \,  \partial_{W_1}^N\cdots\partial_{W_d}^N \partial_{W_m}^{N}u_N(x)\prod_{i=1;\, i\neq m}^d \big(W_i(y_i)-W_i(x_i)\big)\,\, =\\
&(\partial_{W_m}^{N}u_N)^{\ast '}(y_1,\cdots , y_{m-1},y_{m+1},\cdots, y_d)\, ,
\end{align*}
where $(\partial_{W_m}^{N}u_N)^{\ast '}$ denotes the $W-$ interpolation of the function $\partial_{W_m}^{N}u_N$  in the $(d-1)$-dimensional setup.

In this way,
\begin{align*}
&\|\partial_{W_m} u^*_N\|_{L^2_{W_m}(\bb T^d)}^2 =  \int_{\bb T^d}\big(\partial_{W_m} u^*_N(y)\big)^2\, d(y^m\otimes W_m) = \\
 \sum_{x\in \frac1N\bb T^d_N} & \int_{Q_N(x)} \big(\partial_{W_m} u^*_N(y)\big)^2\, d(y^m\otimes W_m) \leq
 \sum_{x\in \frac1N\bb T^d_N} \int_{Q_N(x)} \sum_{z \in \frac 1N \bb T^d_N;\,\, |z-x|=1/N}   (\partial^N_{W_m}u_N(z))^2\;d(y^m\otimes W_m)\; \\
\leq\; &\frac{2^d}{N^{d-1}}\sum_{x\in\frac 1N \bb T^d_N}(\partial^N_{W_m}u_N(x))^2\Big(W_m(\frac{x_m+1}{N})- W_m(\frac{x_m}N)\Big) \ =\;
 2^d\|\partial^N_{W_m} u_N\|_{L^2_{W_m}(\bb T^d_N)}^2.
\end{align*}

Thus we have shown that
$$ \|u^*_N\|_{H_{1,W}(\bb T^d)}\leq 2^d \|u_N\|_{H_{1,W}(\bb T_N^d)}<2^dC, $$
where the last inequality follows from \eqref{L}, and this concludes the proof.
\end{proof}

\vspace{1cm}

We will now obtain the converse procedure, that is, how to use a measurable function $f$ in $L^2(\bb T^d)$ to properly define a mesh function in $\bb T_N^d$. This is not straightforward, since the restriction of $f$ to the set $\{0,1/N,\ldots,(N-1)/N\}^d$ is not well-defined (this is a set of Lebesgue measure zero). 

For arbitrary functions $f\in L^2(\bb T^d)$ and  $g\in L^2_{x^k\x W_k,0}(\bb T^d)$, consider the mesh functions defined on the discrete torus $\frac1N\bb T_N^d$ of the form
\begin{equation}
\label{discrete L2}
f_N(x) = N^d\int_{Q_N(x)} f(y)dy,
\end{equation}
and
\begin{equation}
\label{discrete L21}
g_N(x) = \frac{N^{d-1}}{W_k(x_k+1/N)-W_k(x_k)}\int_{Q_N(x)} g(y)d(y^k\otimes W_k),
\end{equation}
where $ x\in \frac1N\bb T^d_N$, and $Q_N(x) =\{y\in \bb T^d ; 0\leq y_i-x_i< N^{-1}, \   i=1, \ldots, d\}.$

For a generalized function $f \in H^{-1}_{1,W}(\bb T^d)$ in the form \eqref{dual}, $f = f_0 - \sum_{k=1}^d\partial_{x_k}f_k,$ where $f_0\in L^2(\bb T^d),$ and $f_k\in L^2_{x^k\x W_k,0}(\bb T^d)$, we consider
\begin{equation}\label{discrete-functional}
f_N (x) = f_{0,N}(x) - \sum_{k=1}^d \partial^N_{x_k}f_{k,N}(x),
\end{equation}
where $ x\in \frac1N\bb T^d_N$,  $f_{0,N}$ is the discretization given by \eqref{discrete L2}, and $f_{k,N}$ ($k=1,\ldots,d$) is the discretization given by \eqref{discrete L21}. $f_N$ clearly belongs to $H^{-1}_{W}(\bb T_N^d)$.

The following lemma shows that this procedure yields a good approximation for $f$.

\begin{lemma}\label{conv-forte-media}
Let $f\in L^2(\bb T^d)$, and $g\in L^2_{x^j\x W_j}(\bb T^d)$ for some $j=1,2,\ldots, d$. Let $f_N$ and $g_N$ be the mesh functions defined in equations \eqref{discrete L2} and \eqref{discrete L21}.
Let $\widetilde{f}_N$ and $\widetilde{g}_N$ be the piecewise-constant interpolation for $f_N$ and $g_N$. Then,
$$\|f-\widetilde{f}_N\|_{L^2(\bb T^d)}\to 0,$$
and 
$$\|g-\widetilde{g}_N\|_{L^2_{x^j\x W_j,0}(\bb T^d)}\to 0,$$
as $N\to\infty$. In particular, if $F\in H^{-1}_W(\bb T^d)$, then $F_N\to F$ strongly in $H^{-1}_W(\bb T^d)$.
\end{lemma}
\begin{proof}
We will prove the first assertion. The proof of second is analogous. Recall the notation used in the previous lemmas. Since the continuous functions are dense in $L^2(\bb T^d)$  ( and also in $L^2_{x^j\otimes W_j}(\bb T^d)$ ) it is enough to consider $f$ continuous. Thus,
\begin{align*}
\|f-\widetilde{f}_N\|^2_{L^2(\bb T^d)} = \int_{\bb T^d} \Big(  f(y)-\widetilde{f}_N(y)  \Big)^2\, dy \, = \, \sum_{k=1}^{N^d}\int_{Q_N(x^k)}\Big(  f(y)-N^d \int_{Q_N(x^k)}f(z)\, dz  \Big)^2\, dy\, =\\
\sum_{k=1}^{N^d}\int_{Q_N(x^k)}\Big( N^d \int_{Q_N(x^k)}[f(y)-f(z)]\, dz  \Big)^2\, dy\, \leq \,
\sum_{x\in \frac1N\bb T^d_N} N^d\int_{Q_N(x)} \int_{Q_N(x)}[f(y)-f(z)]^2\, dz  \, dy\, , 
\end{align*}
where the last inequality follows from H\"older's inequality. Thus, the previous expression is bounded above by
\begin{align*}
\sum_{x\in \frac1N\bb T^d_N} N^d&\int_{\|\eta\|_\infty \leq 1/N}d\eta \int_{Q_N(x)}[f(y+\eta)-f(y)]^2  \, dy\, \leq 
2^d\sum_{x\in \frac1N\bb T^d_N} \sup_{\substack{{\eta;\, |\eta_i|<1/N}\\ i=1,\ldots , d}} &\int_{Q_N(x)}[f(y+\eta)-f(y)]^2  \,  dy \,.
\end{align*}
To conclude, note that by compacity of $\bb T^d$ the continuous function $f$ is, in fact,  uniformly continuous. So, for a fixed  $\epsilon>0$, there exists $N_0$ such that  $\| \eta\|_\infty <1/N_0$ implies that $|f(y+\eta)-f(y)|< \epsilon^{1/2}/2^d$. Therefore, the previous expression is bounded by $\epsilon$, and the proof of the  Lemma follows.
\end{proof}

\section{Homogenization}\label{sec5}

Our main goal in this Section is to prove the convergence of energies, namely Proposition \ref{hconvergencia}, and also the Homogenization Theorem,  Theorem \ref{homoge}. This last result is presented with fairly general hypotheses on the matrices $A^N$. In Proposition \ref{admitshomog} we provide an example of a very large class of functions that admit homogenization and in subsection \ref{rdo} we consider a scenario in which $A^N$ represents a random environment. We begin with some definitions and auxiliary results.

\subsection{Definitions and auxiliary results}\label{def}
We now focus on the analysis of the asymptotic behavior of the sequence $(u_N)$ given by solutions of the equations
\begin{equation}\label{pd}
\lambda u_N - \nabla^N A^N \nabla_W^N u_N = f^N,  \qquad\,\, \lambda \geq 0
\end{equation}
where  $f^N$ are fixed functions defined on $\frac1N\bb T^d_N$, $\; \nabla^N = (\partial_{x_1}^N,\ldots,\partial_{x_d}^N)$ and $\nabla_W^N = (\partial_{W_1}^N,\ldots,\partial_{W_d}^N)$ are the difference operators, and $A^N = (a_{kk})_{d\times d}$ are diagonal matrices, $A^N = (a_{kk})_{d\times d}$, of order $d$ satisfying the ellipticity condition: there exists a constant $\theta>0$ such that
$\theta^{-1}\le a^N_{kk}(x) \le \theta,$
for every $x\in \bb T^d$ and $k=1,\ldots,d$. 

The continuous counterpart of the theory developed in Subsection \ref{sobdis} can be found in \cite{SV}. More precisely,
one can find results on existence, uniqueness and boundedness of weak solution of the problem
\begin{equation}\label{h4}
\lambda u_0 - \nabla A\nabla_W u_0 = f, \qquad\,\, \lambda \geq 0.
\end{equation}

 We say that the diagonal matrix $A^N=(a_{kk}^N)$ $H$-converges to the diagonal matrix $A = (a_{kk})$,
denoted by $A^N\stackrel{H}{\longrightarrow} A$, if for every sequence $f^N$  of functionals on $H_{1,W}(\bb T_N^d)$
and $f\in  H^{-1}_W(\bb T^d)$, such that $f^N\to f$ as $N\to\infty$ strongly in $H^{-1}_W (\bb T^d)$, we have
\begin{itemize}
\item $u_N \to u_0$ weakly in $H_{1,W}(\bb T^d)$ as $N\to\infty$,
\item $a_{kk}^N \partial_{W_k}^N u_N\to a_{kk} \partial_{W_k} u_0$ weakly in $L^2_{x^k\x W_k}(\bb T^d)$ for each $k=1,\ldots,d$,
\end{itemize}
where $u_N:\frac1N\bb T_N^d\to\bb R$ is the solution of the \eqref{pd} and $u_0\in H_{1,W}(\bb T^d)$ is the unique weak solution of the \eqref{h4}.

In this case, we say that the diagonal matrix $A$ is a \textit{homogenization} of the sequence of random matrices $A^N$. We also say that
the operator $\nabla A\nabla_W$ is a \emph{homogenization} of the sequence of random operators $\nabla^N A^N\nabla_W^N$.

Next, we present an estimate that will be useful in this Section.

\begin{lemma}
\label{limitacaouniforme}
Let $f_N\to f$ as $N\to \infty$ weakly in $H_W^{-1}(\bb T^d)$. Then, if $u_N:\frac1N\bb T_N^d\to\bb R$ is the solution of the \eqref{pd},
there exists $C\geq 0$, not depending on $N$, such that
$$\|u_N\|_{H_{1,W}(\bb T_N^d)}\leq C.$$
\end{lemma}
\begin{proof}
Using lemma \ref{lmdiscreto}, we obtain a constant $C_1$, not depending on $N$ such that
$$\|u_N\|_{H_{1,W}(\bb T_N^d)}\leq C_1\|f_N\|_{H^{-1}_W(\bb T_N^d)}.$$
Note that $f_N\to f$ weakly in $H^{-1}_W(\bb T^d)$ means that $\widetilde{f}_N\to f$ weakly in $H^{-1}_W(\bb T^d)$. Since $\|f_N\|_{H_{W}^{-1}(\bb T_N^d)} = \|\widetilde{f}\|_{H^{-1}_W(\bb T^d)}$, and $\widetilde{f}_N$ is weakly convergent, it is bounded, and thus there exists a constant $C_2\geq 0$, such that 
$$\|f_N\|_{H^{-1}_W(\bb T_N^d)} = \|\widetilde{f}_N\|_{H^{-1}_W(\bb T^d)}\leq C_2.$$
Thus, for all $N\geq 1$,
$$\|u_N\|_{H_{1,W}(\bb T_N^d)}\leq C_1C_2 = C.$$
\end{proof}

We will now prove a very simple version of the compensated compactness Theorem.

\begin{lemma}[Compensated Compactness in $L^2_{x^k\otimes W_k}(\bb T^d)$]\label{compcomp}
For  $k=1,\ldots,d$, let $(g_{k,N})$ and $(w_{k,N})$ be two sequences of functions on $L^2_{x^k\otimes W_k}(\bb T^d)$ such that
$$ g_{k,N}\to g_k,\quad \hbox{strongly in~~}L^2_{x^k\otimes W_k}(\bb T^d)  \qquad \text{and} \qquad w_{k,N}\to w_k,\quad \hbox{weakly in~~}L^2_{x^k\otimes W_k}(\bb T^d),$$
where $g_k, w_k\in L^2_{x^k\otimes W_k}(\bb T^d)$. 
Then,
$$g_{k,N}w_{k,N} \to g_kw_k\quad\hbox{weakly}\hbox{~in~}L^2_{x^k\otimes W_k}(\bb T^d).$$
\end{lemma}
\begin{proof}
Let $\phi\in {\mf D}_W(\bb T^d)$, then
\begin{eqnarray*}
\int_{\bb T^d} g_{k,N}w_{k,N}\phi d(x^k\otimes W_k)= \int_{\bb T^d} (g_{k,N} - g_k)w_{k,N}\phi d(x^k\otimes W_k)
+ \int_{\bb T^d} g_k w_{k,N}\phi d(x^k\otimes W_k).
\end{eqnarray*}

Let us deal with each term in the right-hand side of the previous equation.

Note that
\begin{eqnarray*}
\left|\int_{\bb T^d} (g_{k,N}-g_k))w_{k,N}\phi d(x^k\otimes W_k)\right| &\leq& \| g_{k,N} - g_k\|_{L^2_{x^k\otimes W_k}(\bb T^d)}\| w_{k,N}\phi\|_{L^2_{x^k\otimes W_k}(\bb T^d)}.
\end{eqnarray*}
Note that, from Lemma \ref{right-cont-bounded}, $\phi$ is bounded, and since $(w_{k,N})$ is a weakly convergent sequence, its norm is uniformly bounded. 
Therefore, the previous equation tends to zero as $N\to\infty$.

We now deal with the other term. Begin by recalling that $\phi$ is bounded. Thus, $q_k\phi \in L^2_{x^k\otimes W_k}(\bb T^d)$, since $\bb T^d$ is a set of finite $d(x^k\otimes W_k)$-measure.
Therefore, the weak convergence of $w_{k,N}$ to $w_k$ implies that
$$\int_{\bb T^d} g_k w_{k,N}\phi d(x^k\otimes W_k) = \int_{\bb T^d} w_{k,N} (g_k\phi) d(x^k\otimes W_k)\to \int_{\bb T^d} w_k g_k\phi d(x^k\otimes W_k).$$
This concludes the proof.
\end{proof}

We conclude this subsection with a version of Compensated Compactness Theorem for discrete approximations.
\begin{corollary}[Compensated Compactness for Discrete Approximations]\label{compensated1}
For  $k=1,\ldots,d$, let $(q_{k,N})$ and $(v_{k,N})$ be sequences of functions on $L^2_{x^k\otimes W_k}(\bb T_N^d)$ such that
$$ q_{k,N}\to q_k,\quad \hbox{strongly in~~}L^2_{x^k\otimes W_k}(\bb T^d)  \qquad \text{and} \qquad v_{k,N}\to v_k,\quad \hbox{weakly in~~}L^2_{x^k\otimes W_k}(\bb T^d),$$
where $q_k, v_k\in L^2_{x^k\otimes W_k}(\bb T^d)$. Then,
$$q_{k,N}v_{k,N} \to q_kv_k\quad\hbox{weakly}\hbox{~in~}L^2_{x^k\otimes W_k}(\bb T^d).$$
\end{corollary}
\begin{proof}
Let $w_{k,N} = \widetilde{v}_{k,N}$ and $g_{k,N} = \widetilde{q}_{k,N}$. Thus, from definition, $w_{k,N}$ converges weakly in $L^2_{x^k\otimes W_k}(\bb T^d)$
to $v_k$, and $g_{k,N}$ converges strongly in $L^2_{x^k\otimes W_k}(\bb T^d)$ to $q_k$. Therefore, we may apply lemma \ref{compcomp} to conclude that
$$\widetilde{v_{k,N}q_{k,N}} = \widetilde{v}_{k,N} \widetilde{q}_{k,N} = w_{k,N}g_{k,N} \to v_kq_k\hbox{~~weakly in~~}L^2_{x^k\otimes W_k}(\bb T^d).$$
\end{proof}

\begin{remark}
One should notice that lemma \ref{compcomp} is, indeed, a version of the Compensated Compactness Theorem. In fact, the classical assumptions would be
$g_{k,N}\to g_k$ and $w_{k,N}\to w_k$ weakly in $L^2_{x^k\otimes W_k}(\bb T^d)$, and $\partial_{x_k}g_{k,N} \to h$ strongly in $H^{-1}_W(\bb T^d)$, 
where $\partial_{x_k}g_{k,N}$ should be understood as in lemma \ref{caracterization}.
However, in our setup, the functional induced by $g_{k,N}$ coincides with $-\partial_{x_k}g_{k,N}$, and thus, since $\|\partial_x g_{k,N}\|_{H^{-1}_W(\bb T^d)} = \|g_{k,N}\|_{L^2_{x^k\otimes W_k}(\bb T^d)}$, we recover strong convergence for $g_{k,N}$. 
The reason why the functional induced by $g_{k,N}$ is $-\partial_{x_k}g_{k,N}$ instead of the standard functional is because the standard functional given by
$$\phi \mapsto \int_{\bb T^d} g_{k,N} \phi dx$$
is not well-defined for $g_{k,N}\in L^2_{x^k\otimes W_k}(\bb T^d)$. In fact, $g_{k,N}$ may not belong to $L^2(\bb T^d)$.
\end{remark}

\subsection{Main Results} \label{main}
We are now in a position to state and prove the homogenization of the difference operators introduced in the previous Section.  We begin by proving an auxiliary lemma that will be needed in the proof of the convergence of energies of a sequence of homogenized matrices.

\begin{lemma}\label{lemachapeu}
Let $X$ be a Banach space and $X^\ast$ its dual. If $f_N\in X^\ast$ is such that $f_N\to f \in X^\ast$, and $u_N\in X$ is such that $u_N\to u\in X$ weakly. 
Then,
$$f_N(u_N)\to f(u),$$
as $N\to\infty$.
\end{lemma}
\begin{proof}
Since $u_N$ converges weakly in $X$, it forms a bounded sequence, that is, there exists $C\geq 0$ such that, for each $N\geq 1$, $\|u_N\|\leq C$. 
We also have
$$|f_N(u_N) - f(u_N)|\leq \|f_N-f\|_{X^\ast}\|u_N\|_X\leq C\|f_N-f\|_{X^\ast},$$
which tends to zero as $N\to\infty$. On the other hand, from weak convergence, $f(u_N)\to f(u)$ as $N\to\infty$. Therefore, $f_N(u_N)\to f(u)$ as $N\to\infty$.
\end{proof}

The next proposition shows that even though the $H$-convergence only requires weak convergence in its definition, it yields a convergence in a strong sense (convergence in the $L^2$-norm for the piecewise-constant interpolation).

\begin{proposition}\label{hconvergencia}
Let $A^N\stackrel{H}{\longrightarrow} A$, as $N\to\infty$, with $u_N$ being the solution of \eqref{pd},
where $f\in H_W^{-1}(\bb T^d)$ is fixed, $f^N\to f$ strongly in $H^{-1}_W(\bb T^d)$ and, $u_0$ is the weak solution of \eqref{h4}. Then, the following limit relations hold true:
$${u}_N \to u_0 \hbox{~in~}{L^2(\bb T^d)},$$
$$\frac{1}{N^d}\sum_{x\in\bb T_N^d} u_N^2(x/N) \stackrel{N\to\infty}{\longrightarrow} \int_{\bb T^d} u_0^2(x) dx,\quad\hbox{and}$$
%\begin{align*}
$$\frac{1}{N^{d-1}}\sum_{k=1}^d\sum_{x\in\bb T_N^d} a_{kk}^N(x/N) (\partial_{W_k}^N u_N(x/N))^2 \left[W_k((x_k+1)/N)-W_k(x_k/N)\right]
 \stackrel{N\to\infty}{\longrightarrow} \sum_{k=1}^d \int_{\bb T^d} a_{kk}(x) (\partial_{W_k} u_0(x))^2 d(x^k\x W_k).$$
%\end{align*}
\end{proposition}

\begin{proof}
We begin by proving that
\begin{equation}\label{hconv1}
 f^N (u_N)\to f(u_0),
\end{equation}
as $N\to\infty$. As we plan on using Lemma \ref{fraca-interp}, we need to obtain a bound. By Lemma \ref{limitacaouniforme}, the sequence $u_N$ is $\|\cdot\|_{1,W}$ bounded 	
uniformly.  In particular, there exist constants $C_1,C_2^k>0$, $k=1,\ldots,d$, such that, for all $N\geq 1$, we have
\begin{equation}
\label{desigbound}
\frac{1}{N^d} \sum_{x\in\bb T_N^d} u_N^2(x/N) \leq C_1,
\end{equation}
and
\begin{equation}
\label{desigbound2}
\frac{1}{N^{d-1}}\sum_{x\in\bb T_N^d} (\partial_{W_k}^N u_N(x/N))^2(W_k((x_k+1)/N)-W_k(x_k/N)) \leq C_2^k.
\end{equation}

Now, observe that from Lemma \ref{caracterization discrete}, there exist functions $f_{0,N}, \ldots, f_{d,N}$ such that
$$f_N(u_N) = \<f_{0,N},u_N\>_{L^2(\bb T_N^d)} + \sum_{k=1}^d \<f_{k,N},\partial_{W_k}^Nu_N\>_{L^2_{x^k\otimes W_k}(\bb T^d)}.$$

This motivates us to define the following functionals $g_{i,N}\in H^{-1}_W(\bb T_N^d)$ by $g_{0,N}(v) = \<f_{0,N},v\>_{L^2(\bb T_N^d)}$,
and $g_{k,N}(v) = \<f_{k,N},\partial_{W_k}^Nv\>_{L^2_{x^k\otimes W_k}(\bb T_N^d)}$, for $k=1,\ldots,d$. Note that
$$g_{0,N}(v) = \<\widetilde{f}_{0,N},\widetilde{v}\>_{L^2(\bb T^d)}, \qquad \text{and} \qquad \,\, g_{k,N}(v) = \<\widetilde{f}_{k,N},\widetilde{\partial_{W_k}^Nv}\>_{L^2_{x^k\otimes W_k}(\bb T^d)}.$$
We have, by hypothesis, that $f_N\to f$ strongly, which means
$$\|\widetilde{f}_N-f\|^2_{H^{-1}_W(\bb T^d)} = \|\widetilde{f}_{N,0}-f_0\|_{L^2(\bb T^d)}^2 + \sum_{k=1}^d \|\widetilde{f}_{k,N}-f_k\|_{L^2_{x^k\otimes W_k}(\bb T^d)}^2\to 0,$$
as $N\to \infty$. Thus, $g_{0,N}\to f_0$ strongly in $L^2(\bb T^d)$, $g_{k,N}\to f_k$ strongly in $L^2_{x^k\otimes W_k}(\bb T^d)$. 

Now, observe that $u_N\to u_0$ weakly in $H^1_W(\bb T^d)$ means that $u_N^\ast \to u_0$ weakly in $H^1_W(\bb T^d)$, which in turn implies that
\begin{equation}\label{convfraca1}
u_N^\ast \to u_0,\qquad \hbox{weakly in}\quad L^2(\bb T^d),
\end{equation}
and
\begin{equation*}
\partial_{W_k} u_N^\ast \to \partial_{W_k} u_0,\qquad \hbox{weakly in}\quad L_{x^k\otimes W_k}^2(\bb T^d).
\end{equation*}
Nevertheless, by equation \eqref{deriv-interp}, we have
$$\partial_{W_k} u_N^\ast = \left(\partial_{W_k}^N u_N\right)^{(k)},$$
and thus
\begin{equation}\label{convfraca2}
\left(\partial_{W_k}^N u_N\right)^{(k)} \to \partial_{W_k} u_0,\qquad \hbox{weakly in}\quad L_{x^k\otimes W_k}^2(\bb T^d).
\end{equation}

Using Lemma \ref{fraca-interp}, \eqref{convfraca1} implies
\begin{equation}\label{convfraca3}
\widetilde{u}_N \to u_0,\qquad \hbox{weakly in}\quad L^2(\bb T^d),
\end{equation}
and \eqref{convfraca2} implies
\begin{equation}\label{convfraca4}
\widetilde{\partial_{W_k}^Nu_N}\to \partial_{W_k}u_0,\qquad \hbox{weakly in}\quad L_{x^k\otimes W_k}^2(\bb T^d).
\end{equation}

Now, since $\widetilde{g}_{0,N}\to f_0$ strongly in $L^2(\bb T^d)$ and $\widetilde{g}_{k,N}\to f_k$ strongly in $L^2_{x^k\otimes W_k}(\bb T^d)$, we may apply Lemma \ref{lemachapeu} to equations \eqref{convfraca3} and \eqref{convfraca4} to obtain that
\begin{equation}\label{convfunc1}
\widetilde{g}_{0,N}(\widetilde{u}_N) \to f_0(u_0),\qquad \text{as}\,\, N\to\infty \,\,\text{and,}
\end{equation}

\begin{equation}\label{convfunc2}
\widetilde{g}_{k,N}(\widetilde{\partial_{W_k}^Nu_N})\to f_k(\partial_{W_k}u_0), \qquad \text{as}\,\, N\to\infty \, .
\end{equation}

Note that
$$f(u_0) = f_0(u_0) + \sum_{k=1}^d f_k(\partial_{W_k}(u_0)),\qquad \text{and}$$
\begin{eqnarray*}
f_N(u_N)= g_{0,N}(u_N) + \sum_{k=1}^d g_{k,N}(\partial_{W_k}^N(u_N))
= \widetilde{g}_{0,N}(\widetilde{u}_N) + \sum_{k=1}^d \widetilde{g}_{k,N}(\widetilde{\partial_{W_k}^Nu_N}).
\end{eqnarray*}
Thus, from \eqref{convfunc1} and \eqref{convfunc2}, $f_N(u_N)$ converges to $f(u_0).$ Now, note that
$$f(u_0) = \lambda\int_{\bb T^d} u_0^2 dx + \sum_{k=1}^d \int_{\bb T^d} a_{kk} (\partial_{W_k} u_0)^2 d(x^k\x W_k),$$
since, by the hypothesis that $A\stackrel{H}{\longrightarrow} A$, $u_0$ is the weak solution of $\lambda u_0 - \nabla A\nabla_W u_0 = f.$ Note also that
\begin{eqnarray*}
f^N (u_N) & = & \frac{1}{N^d}\sum_{x\in\bb T_N^d} (\lambda u_N(x/N) -\nabla^N A^N\nabla_W^N u_N(x/N)) u_N(x/N)\\
&=&  \frac{\lambda}{N^d}\sum_{x\in\bb T_N^d} u_N^2(x/N) - \frac{1}{N^d}\sum_{x\in\bb T_N^d} u_N(x/N) \nabla^N A^N\nabla_W^N u_N(x/N),
\end{eqnarray*}
which, after a summation by parts in the above expressions, and using that $f_N(u_N)\to f(u_0)$, we obtain

\begin{align}\label{hconv2}
\frac{\lambda}{N^d}\sum_{x\in\bb T_N^d} u_N^2(x/N) + \frac{1}{N^{d-1}}\sum_{k=1}^d & \sum_{x\in\bb T_N^d} a_{kk}^N (\partial_{W_k}^N u_N(x/N))^2 [W_k((x_k+1)/N)-W_k(x_k/N)]     \nonumber \\
& \stackrel{N\to\infty}{\longrightarrow} \lambda \int_{\bb T^d} u_0^2 dx + \sum_{k=1}^d \int_{\bb T^d} a_{kk} (\partial_{W_k} u_0)^2 d(x^k\x W_k).
\end{align}

 Suppose that $u_N$ does not converge to $u_0$ in $L^2(\bb T^d)$. That is, there exist $\epsilon>0$ and a subsequence $(u_{N_j})$ such that
$$\|\widetilde{u}_{N_j} - u_0\|_{L^2(\bb T^d)}>\epsilon,$$
for all $j$. By Lemma \ref{compact-interp}, we have that there exists $v\in L^2(\bb T^d)$ and a further subsequence (also denoted by $u_{N_j}$) such that
$$u_{N_j}^\ast \stackrel{j\to\infty}{\longrightarrow} v,\quad\hbox{~in~} L^2(\bb T^d).$$
Using Lemma \ref{strong-interp}, we further obtain that
$$\widetilde{u}_{N_j} \stackrel{j\to\infty}{\longrightarrow} v,\quad\hbox{~in~} L^2(\bb T^d).$$
This implies that
$$\widetilde{u}_{N_j} \to v, \quad \hbox{weakly in~}L^2(\bb T^d),$$
but this is a contradiction. Indeed, from $H$ convergence of $A_N$ to $A$, we have that $u_N\to u_0$ weakly in $H_{1,W}(\bb T^d)$,
which means that $u_N^\ast \to u_0$ in $H_{1,W}(\bb T^d)$. Thus $u_N^\ast\to u_0$ weakly in $L^2(\bb T^d)$. Using Lemma \ref{fraca-interp}, this implies that $\widetilde{u}_N\to u_0$ weakly in $L^2(\bb T^d)$. In particular,
$$u_{N_j} \to u_0,\quad \hbox{weakly in~}L^2(\bb T^d),$$
and at the same time $\|v-u_0\|_{L^2(\bb T^d)} \geq \epsilon.$ Therefore, $u_N \to u_0$ in $L^2(\bb T^d)$. The proof thus follows from expression \eqref{hconv2}.
\end{proof}
\begin{corollary}\label{suave-hconv}
Let $u_0\in \mf D_W(\bb T^d)$ and $u_N:\frac1N \bb T^d_N\to \bb R$ such that $\lim u_N = u_0$ in $ L^2(\bb T^d)$.
%$$u_N \to u_0 \hbox{~in~} L^2(\bb T^d)$$
Then,
$$\|u_0 - u_N\|_{L^2(\bb T_N^d)} \longrightarrow 0\, .$$
\end{corollary}
\begin{proof}
We have that $u_N\to u_0$ in $L^2(\bb T^d)$ means that
\begin{equation}\label{a}
\|\widetilde{u}_N-u_0\|_{L^2(\bb T^d)}\longrightarrow 0,
\end{equation}
as $N\to\infty$. On the other hand, from the definition of the set $Q_N(x)$ in Subsection \ref{correspondencia} and, the fact that
the functions in $\mf D_W(\bb T^d)$ are right-continuous (see Remark \ref{DWrightcontinuous}), we have that
$$\widetilde{u_0\mid_{\bb T_N^d}} \longrightarrow u_0$$
pointwise as $N\to\infty$. From Lemma \ref{right-cont-bounded}, $u_0$ is bounded, and thus $\widetilde{u_0\mid_{\bb T_N^d}}$ is bounded, and therefore integrable ($\bb T^d$ has finite Lebesgue measure). Thus, we can use the Dominated Convergence Theorem to conclude that 
\begin{equation}\label{b}
\|\widetilde{u_0\mid_{\bb T_N^d}} - u_0\|_{L^2(\bb T^d)} \longrightarrow 0,
\end{equation}
as $N\to\infty$. Therefore, using \eqref{a} and \eqref{b}, we obtain
\begin{eqnarray*}
\|u_N-u_0\|_{L^2(\bb T_N^d)} &=& \|\widetilde{u}_N-\widetilde{u_0\mid_{\bb T_N^d}}\|_{L^2(\bb T^d)}\\
&\leq& \|\widetilde{u}_N-u_0\|_{L^2(\bb T^d)} + \|u_0 - \widetilde{u_0\mid_{\bb T_N^d}}\|_{L^2(\bb T^d)} \, \longrightarrow 0,
\end{eqnarray*}
as $N$ goes to  $\infty$.
\end{proof}

\vspace{1cm}

We will now state and prove the main result of this paper.

\begin{theorem}
\label{homoge}
Let $A^N = (a^N_{kk})_{k=1:d}$ be a sequence of diagonal matrices and $\theta>0$, such that $\theta^{-1}\leq a^N_{kk}\leq \theta$,   $a^N_{kk}$ and $1/a^N_{kk}$ converges weakly in $L^2_{x^k\otimes W_k}(\bb T^d)$. Then, $A^N$ admits a homogenization.
% where the homogenized matrix $A$ does not depend on the realization $\omega$.
\end{theorem}

\begin{proof}[Proof of Theorem \ref{homoge}]
Fix $f\in H_W^{-1}({\bb T}^d)$, and consider the problem
\begin{equation}\label{mais}
\lambda u_N - \nabla^N A^N \nabla_W^Nu_N = f_N,
\end{equation}
where $f_N$ is the discretization obtained in equation \eqref{discrete-functional}. From Lemma \ref{conv-forte-media}, $f_N\to f$ strongly in $H^{-1}_W(\bb T^d)$. 

Using Lemma \ref{limitacaouniforme}, there exists a unique weak solution $u_N$ of the problem above such that its $H_{1,W}^N$-norm is uniformly bounded in $N$. That is, there exists a constant $C>0$ such that
$$\|u_N\|_{H_{1,W}(\bb T_N^d)} \leq C .$$
From Lemma \ref{compact-interp}, there exists a convergent subsequence of $u_N$ (which we will also denote by $u_N$) such that
$$u_N\to u, \qquad \hbox{weakly in}\quad {H}_{1,W}({\bb T}^d).$$

In particular,

\begin{equation}\label{h2}
\partial_{W_k}^N u_N\stackrel{N\to\infty}{\longrightarrow}\partial_{W_k} u\quad\hbox{weakly in}\quad L^2_{x^k\x W_k}(\bb T^d).
\end{equation}

Applying \eqref{mais} to $u_N$, we obtain
$$\lambda \|u_N\|_{L^2(\bb T_N^d)}^2 + \sum_{k=1}^d \|a_{kk}^N\partial_{W_k}^Nu_N\|_{L^2_{x^k\otimes W_k}(\bb T_N^d)}^2 = f_N(u_N) \leq \|f_N\|_{H^{-1}_W(\bb T_N^d)}\|u_N\|_{H^1_W(\bb T_N^d)}.$$
Thus, for each $k=1,\ldots, d$, and using Lemma \ref{limitacaouniforme}, we have
$$\|a_{kk}^N \partial_{W_k}^Nu_N\|_{L^2_{x^k\otimes W_k}(\bb T_N^d)}^2 \leq \|f_N\|_{H^{-1}_W(\bb T_N^d)} \cdot C.$$
Using the same argument (from weak convergence of the functionals) we can find a constant $C_1\geq 0$ such that $\|f_N\|_{H_W^{-1}(\bb T_N^d)}\leq C_1$. Therefore,
$$\|a_{kk}^N\partial_{W_k}^Nu_N\|_{L^2_{x^k\otimes W_k}(\bb T_N^d)} \leq C_1 C.$$
This, in turn, implies that
$$\|a_{kk}^N\widetilde{\partial_{W_k}^Nu_N}\|_{L^2_{x^k\otimes W_k}(\bb T^d)} = \|a_{kk}^N\partial_{W_k}^Nu_N\|_{L^2_{x^k\otimes W_k}(\bb T_N^d)}$$
is uniformly bounded in $N$. Thus, since $L^2_{x^k\otimes W_k}(\bb T^d)$ is a separable Hilbert space, we can find a further subsequence (also denoted by $u_N$), such that
\begin{equation}\label{h1}
a_{kk}^N \partial_{W_k}^N u_N\to v_{0,k}\qquad \hbox{weakly in}\quad L^2_{x^k\x W_k}({\bb T}^d),
\end{equation}
as $N\to\infty$, $v_{0,k}$ being some function in $L^2_{x^k\x W_k}(\bb T^d)$.

 Since $(u_N)$ is uniformly bounded in the Sobolev-norm and  $L^2(\bb T^d)$ is precompact in this space we have
$$u_N \to u \qquad \text{strongly in}\,\, L^2(\bb T^d)\, .$$
 In particular, 
$$u_N \to u \qquad \text{strongly in}\,\,  H^{-1}_W(\bb T^d).$$ 
On the other hand, $(\lambda u_N- \nabla^N A^N \nabla_W^Nu_N)$
converges strongly (to $f$) in $H^{-1}_W(\bb T^d)$. Therefore,
$$\nabla^N A^N \nabla_W^N u_N\to v_{0}\qquad \text{strongly in} \,\,H^{-1}_W(\bb T^d).$$ From the very definition of the functional $\nabla^N A^N \nabla_W^Nu_N$, the previous convergence means that
for each $k$,
$$a_{kk}^N \partial_{W_k}^N u_N \to v_{0,k}\hbox{~~strongly in~~}L^2_{x^k\x W_k}({\bb T}^d).$$

 Denote by $b_k\in L^2_{x^k\x W_k}({\bb T}^d)$ the  weak limit of the sequence $(1/a^N_{kk})$:
$$1/a_{kk}^N \to b_k\hbox{~~weakly in~~}L^2_{x^k\x W_k}({\bb T}^d).$$

Since $\theta^{-1} < a_{kk}^N(x) < \theta$, we have that  $1/a_{kk}^N$ is uniformly bounded, and $1/a_{kk}^N>\theta$. Further, $b_k>0$ $d(x^k\otimes W_k)$-a.e. . In fact, if there exists a measurable set $A$ with positive $d(x^k\otimes W_k)$
measure, such that $b_k=0$ in $A$,  take the function $\phi = 1_A$. Then,
$$0 < \theta \hbox{measure}(A) \leq \int \widetilde{\frac{1}{a_k^N}} \phi d(x^k\otimes W_k) \to \int b_k 1_A d(x^k\otimes W_k) = 0.$$
A contradiction. Thus $b_k>0$.\\

From the Compensated Compactness Theorem (take $q_k^N = a_{kk}^N \partial_{W_k}^N u_N$ and $v_k^N = 1/a_{kk}^N$ in Corollary \ref{compensated1}):
$$\frac{1}{a_{kk}^N} a_{kk}^N \partial_{W_k}^N u_N \to b_k v_{0,k}, \qquad \text{weakly in} \,\, L^2_{x^k\x W_k}({\bb T}^d).$$
On the other hand,
$$\frac{1}{a_{kk}^N} a_{kk}^N \partial_{W_k}^N u_N = \partial_{W_k}^N u_N   \to \partial_{W_k}u, \qquad \text{weakly in}\,\, L^2_{x^k\x W_k}({\bb T}^d)\,. $$
From uniqueness of the weak limit, we have that
$\partial_{W_k}u = b_k v_{0,k}.$
Since $b_k\neq 0$, we have that
$$v_{0,k} = \frac{1}{b_k} \partial_{W_k}u.$$

Thus, we can summarize our findings:
\begin{eqnarray*}
u_N\to u \qquad \text{strongly in} \,\, L^2_{x^k\otimes W_k}(\bb T^d)\, , \\ 
\partial_{W_k}u_N\to \partial_{W_k}u  \qquad \text{weakly in}\,\, L^2_{x^k\otimes W_k}(\bb T^d), \qquad \text{and}\,  \\
a_k^N\partial_{W_k}u_N \to \frac{1}{b_k}\partial_{W_k}u \qquad \text{strongly in} \,\, L^2_{x^k\otimes W_k}(\bb T^d)\,.
\end{eqnarray*}

Therefore,
$u$ solves the problem
$$\lambda u - \nabla A\nabla_W u = f,$$
where $A$ is the diagonal matrix with entries given by $1/b_k$.

To conclude the proof it remains to be shown that we can pass from the subsequence to the sequence. This follows from uniqueness of weak solutions of the problem \eqref{h4}, see \cite[Proposition 3.4]{SV}. The fact that any converging subsequence is a solution to the same problem, and the fact that
$u_N$ is uniformly bounded in the Sobolev norm, thus we can find a convergent subsequence (thus a sequence that do not converge to the solution, must converge to somewhere else, since the limit point is also a solution, uniqueness shows the result).
\end{proof}

We will now provide an example of a very large class of functions that admit homogenization. Recall the definition of the space $\mathbb{M}_W$ given
below Remark \ref{nablaAnablaW}.

\begin{proposition}\label{admitshomog}
Let $A = (a_{kk}) \in\mathbb{M}_W$ be a diagonal matrix such that $\theta^{-1}\leq a_{kk}$ for some  $\theta>0$ and,  the discretization $A^N = (a_{kk}^N)$ be the sequence of diagonal matrix  obtained from \eqref{discrete L21}.
Then, the sequence $(A^N)$ admits a homogenization.
\end{proposition}
\begin{proof}
It is clear that $0<\theta^{-1}\leq a_{kk}^N$. Further, it is clear that the right-continuity implies the pointwise convergence
of $\widetilde{a}_{kk}^N$ to $a_{kk}$. Finally, from Lemma \ref{right-cont-bounded}, the functions $a_{kk}$ are bounded, and thus the sequences $a_{kk}^N$ are bounded.
From the Dominated Convergence Theorem
$$a_{kk}^N \to a_{kk} \hbox{~~strongly in~~}L^2_{x^k\otimes W_k}(\bb T^d).$$
On the other hand, we also have the pointwise convergence of $\widetilde{1/a_{kk}^N}$ to $1/a_{kk}$, and the bound $1/a_{kk}^N\leq \theta$ implies
that
$$1/a_{kk}^N \to 1/a_{kk} \hbox{~~strongly~~}L^2_{x^k\otimes W_k}(\bb T^d).$$
Therefore, the result follows from Theorem \ref{homoge}.
\end{proof}

\subsection{Homogenization of Random difference operators}\label{rdo}

In this subsection we consider the homogenization problem when the matrix $A^N$ represents a random environment. More precisely, we focus on the analysis of the asymptotic behavior of the sequence $(u_N)$ given by solutions of the equations
\begin{equation*}
\lambda u_N - \nabla^N A^N \nabla_W^N u_N = f^N,
\end{equation*}
where $f^N$ are fixed functions defined on $\frac1N\bb T^d_N$, $\; \nabla^N = (\partial_{x_1}^N,\ldots,\partial_{x_d}^N)$ and $\nabla_W^N = (\partial_{W_1}^N,\ldots,\partial_{W_d}^N)$ are the difference operators and, the random diagonal matrix $A^N = (a_{kk})_{d\times d}$ of order $d$ represent the statistically homogeneous rapidly oscillating coefficients.
Therefore these equations are driven by the random difference operators $\nabla^N A^N \nabla_W^N$, and to fully understand them, we need to understand the random matrices $A^N$.
Thus, let $(\Omega,\mc F, \mu)$ be a standard probability space and $\{ T_x : \Omega \to \Omega; x\in \bb Z^d\}$
be a group of $\mc F$-measurable and ergodic transformations which preserve the measure $\mu$:
\begin{itemize}
\item $ T_x : \Omega \to \Omega$ is $\mc F$-measurable for all $x \in \bb Z^d$,
\item $\mu(T_x \textbf{A}) = \mu(\textbf{A})$, for any $\textbf{A} \in \mc F$ and $x\in \bb Z^d$,
\item $T_0 = \textit{I}\;,\; \;T_x\circ T_y = T_{x+y}$,
\item Any  $f\in L^1(\Omega)$ such that $f(T_x\omega)=f(\omega)\;\; \mu$-a.e.,  for each $x\in \bb Z^d$,
is equal to a constant $\mu$-a.e..
\end{itemize}
Note that the last condition implies that the group $T_x$ is ergodic. We call the underlying probability space $(\Omega,\mc F, \mu)$ \emph{random environment},
and a point $\omega \in \Omega$ a \emph{realization} of the random environment.

Let us now introduce the vector-valued $\mc F$-measurable functions $\{b_{kk}(\omega) ; k=1,\ldots, d\}$ such that there exists
$\theta >0$ with 
\begin{equation}\label{unif-elip}
\theta^{-1} \le b_{kk}(\omega)\le \theta,
\end{equation}
for all $\omega \in \Omega$ and $ k= 1,\ldots, d$.  Then, define the \emph{random} diagonal matrices $B^N$ whose elements  are given by
\begin{equation}
\label{AN}
b^N_{kk}(x):= b_{kk}(T_{Nx}\omega)\;,\;\;x\in \frac{1}{N}\bb T^d_N\;,\;\;k = 1, \ldots , d.
\end{equation}

Let us show  some weak convergences associated the random environment $(b_{kk}^N)$ defined in \eqref{AN}.
First, note that by Birkhoff Ergodic Theorem, we have,
\begin{equation}
\label{bir}
b_{kk}^N\longrightarrow E[b_{kk}] \qquad \text{weakly in}\, \, L^2(\bb T^d)\qquad \text{a.s.,}
\end{equation}
for $k = 1, \dots, d$. We will need an similar result for $L^2_{x^k\otimes W_k}(\bb T^d)$.

Denote by $\mu_{W_k}$ the measure induced by function $W_k$.
By Lebesgue decomposition, there exist, function $g$ such that,
$$\mu_{W_k} = g\lambda + \lambda^\perp $$
where $g\lambda$ and $\lambda^\perp $ are singular  measures and  $\lambda$ denotes de Lebesgue measure.
Let $V_k\subset \bb T$ be the support of $\lambda^\perp $ and  
${\mf V}^k = \bb T\times\ldots \times  \bb T\times V_k\times \bb T \times \ldots\times \bb T\subset \bb T^d$, $V_k$ in the $k$-th component.

Define $a^N_k: \frac1N\bb T^d_N \to \bb R$ as
\begin{equation}\label{ANN}
a^N_{kk}(x) \;=\;
\begin{cases}
b^N_{kk}(x) & \text{ if } {\mf V}^k\cap Q_N(x)= \emptyset,\\
E[b_{kk}] & \text{ if } {\mf V}^k\cap Q_N(x)\neq \emptyset 
\end{cases}
\end{equation}

Here $Q_N(x)$ denotes the cell in $\bb T^d$ with vertex in $x\in \frac 1N\bb T^d_N$,
$$Q_N(x) = \{ y\in \bb T^d;\, 0\leq y_i-x_i\leq 1/N\}\, . $$

Define the random diagonal matrices $A^N = (a^N_{kk})$ and consider the problem
\begin{equation}\label{pdr}
\lambda u_N - \nabla^N A^N \nabla_W^N u_N = f^N,
\end{equation}

The main result of this subsection is

\begin{theorem}\label{homoge-random}
Let $A^N$ be a sequence random matrices, as defined previously in \eqref{pdr}. Then, almost surely, $A^N(\omega)$ admits a homogenization, where the homogenized matrix $A$ does not depend on the realization $\omega$.
\end{theorem}

\begin{proof}The proof follows from Lemma \ref{cre} that ensures the hypothesis of the Theorem \ref{homoge} are valid for this sequence.
\end{proof}

We conclude this subsection with following result.
\begin{lemma}\label{cre}
Let $a^N_{kk}: \frac1N\bb T^d_N \to \bb R$ be as defined above. Then,
\begin{equation*}
a^N_{kk} \longrightarrow E[b_{kk}] \qquad \text{weakly in}\, \, L^2_{x^k\otimes W_k}(\bb T^d)\qquad \text{a.s.,}
\end{equation*}
and
\begin{equation*}
1/a^N_{kk} \longrightarrow B_{kk} \qquad \text{weakly in}\, \, L^2_{x^k\otimes W_k}(\bb T^d)\qquad \text{a.s.,}
\end{equation*}
where the function $B_{kk}$ is given by
$$B_{kk}(x) \;=\;
\begin{cases}
E[1/b_{kk}] & \text{ if } {\mf V}^k\cap Q_N(x)= \emptyset,\\
1/E[b_{kk}] & \text{ if } {\mf V}^k\cap Q_N(x)\neq \emptyset .
\end{cases}
 $$
\end{lemma}
\begin{proof}
Following the notation introduced above, 
let ${\mf V}^k_N = \bigcup_{ {\mf V}^k\cap Q_N(x)\neq \emptyset } Q_N(x)$. We have,
$$ 1_{{\mf V}^k_N} \to 1_{{\overline {\mf V}}^k} \qquad \text{pointwise}.$$
Therefore,
\begin{equation}\label{ind}
1_{{\mf V}^k_N} \to 1_{{\overline {\mf V}}^k} \qquad \text{strongly in}\,\, L^2(\bb T^d),
\end{equation}
where ${\overline {\mf V}}^k$ stands for the closure of the set ${\mf V}^k$. Let $\phi \in \mf  D_W(\bb T^d)$ be fixed. By Lebesgue decomposition we have,
 \begin{equation*}
 \int_{\bb T^d}\widetilde{a^N_{kk}}\phi d(x^k\otimes W_k) = \int_{\bb T^d}\widetilde{a^N_{kk}}\phi gdx + \int_{\bb T^d}\widetilde{a^N_{kk}}\phi d(x^k\otimes \lambda^\perp). 
 \end{equation*}
 
 Note that the support of the measure $d(x^k\otimes \lambda^\perp)$ is confined in the set ${\mf V}^k$, defined above. Since ${\mf V}^k\subset {\mf V}^k_N$, we have that $\widetilde{a^N_{kk}}$ is almost everywhere constant, namely $\widetilde{a^N_{kk}} = E[b_{kk}]$, with respect to the measure $d(x^k\otimes \lambda^\perp)$. Thus, the second integral in the right-hand side in previous expression is equal to
 $$ \int_{\bb T^d}E[b_{kk}]\phi d(x^k\otimes \lambda^\perp)\;. $$

For other side, by \eqref{bir} and \eqref{ind}, we have
\begin{equation}\label{prod}
b_{kk}^N 1_{\left[{\mf V}^k_N\right]^c} \longrightarrow E[b_{kk}] 1_{\left[{\overline{\mf V}^k}\right]^c}  \qquad \text{weakly in}\, \, L^2(\bb T^d)\qquad \text{a.s.,} 
\end{equation}
here $X^c$ denotes the complementary set of $X$.

So, the first integral in the right-hand side in above expression is equal to
\begin{eqnarray*}
\int_{\bb T^d}\widetilde{a^N_{kk}}1_{{\mf V}^k_N}\phi gdx \, +\,  \int_{\bb T^d}\widetilde{a^N_{kk}}1_{\left[{\mf V}^k_N\right]^c}\phi gdx \, =\,
\int_{\bb T^d}E[b_{kk}]1_{{\mf V}^k_N}\phi gdx +  \int_{\bb T^d}\widetilde{b_{kk}^N}1_{\left[{\mf V}^k_N\right]^c}\phi gdx
\end{eqnarray*}
From previous convergence, \eqref{ind} and \eqref{prod},   the right hand-side in previous expression converges to
$$ \int_{\bb T^d}E[b_{kk}]1_{\overline{{\mf V}}^k}\phi gdx +  \int_{\bb T^d}E[b_{kk}]1_{\left[\overline{{\mf V}}^k\right]^c}\phi gdx  \,=\, \int_{\bb T^d}E[b_{kk}]\phi gdx\, .$$
Then, we have showed that
$$\lim_{N\to\infty} \int_{\bb T^d}\widetilde{a^N_{kk}}\phi d(x^k\otimes W_k) =  \int_{\bb T^d}E[b_{kk}]\phi d(x^k\otimes \lambda^\perp)\, +\,  \int_{\bb T^d}E[b_{kk}]\phi gdx \,=\, \int_{\bb T^d}E[b_{kk}]\phi d(x^k\otimes W_k) $$
and this concludes the proof of the first statement.

The second statement follows directly from the first by observing that $1/b_{kk}$ is also an ergodic sequence, and everything we did above may also be done
to the sequence $1/a_{kk}^N$.
\end{proof}

\section{Application}\label{aplicacao-limite}
To conclude the paper we will provide an application of a new result on probability theory which is an improvement of the result
obtained in \cite{v} in two directions: first, it considers a more general model; second, it has a more natural and simpler proof.
It is also noteworthy that the homogenization results obtained in this paper were the key results in proving the main Theorem in the article \cite{FSV}.
For each choice of diagonal matrix function $A^N$ satisfying the hypotheses of Theorem \ref{homoge}, we have a
corresponding version of hydrodynamic limit. For instance, we may obtain a version in random environment using Theorem \ref{homoge-random},
and also a version with any diagonal matrix function $A\in \mathbb{M}_W$.

%
%The result we will prove is a hydrodynamic limit for \textit{gradient processes with conductances in generical environment, for instance, in random environment or more }. The model
%considered by \cite{v} is a particular case of the model we are considering since it does not possess a random environment.

\subsection{The hydrodynamic limit} \label{HL}

We will now use the standard vocabulary in probability theory, that is, càdlàg functions means right-continuous functions
with left limits; tightness is a property regarding compactness; weak convergence is actually weak$^*$ convergence. The reader is also
referred to \cite{kl} and references therein.

We will begin by recalling some definitions. Recall in \eqref{w} the definition of function $W$, consider a sequence of  operators $\nabla^N A^N \nabla_W^N$, that satisfies the hypothesis of the Theorem \ref{homoge}. We will consider, for instance, the random environment introduced in subsection \ref{rdo}.
% and fix a typical realization $\omega \in \Omega$ of the random environment.
 For each $x\in \bb T^d_N$ and $j = 1,\ldots, d$,
define the symmetric rate $\xi_{x, x+e_j}=\xi_{x+e_j, x}$ by
\begin{equation}\label{rate}
\xi_{x, x+e_j} \;=\; \frac{a^N_j(x/N)}{N[W((x+e_j)/N) - W(x/N)]}\;=\;
\frac{a^N_j(x/N)}{N[W_j((x_j+1)/N) - W_j(x_j/N)]}.
\end{equation}
where $a^N_j(x)$ is given by \eqref{ANN}, and  ${e_1,\ldots ,e_d}$ is the canonical basis of $\bb R^d$. Also, let
$b> -1/2\;$ and
\begin{equation*}
c_{x,x+e_j}(\eta) \;=\; 1 \;+\; b \{ \eta(x-e_j) + \eta(x+2 e_j)\}\;,
\end{equation*}
where all sums are modulo $N$.

Distribute particles on $\bb T^d_N$ in such a way that each site
of $\bb T^d_N$ is occupied at most by one particle. Denote by $\eta$ the
configuration of the state space $\{0,1\}^{\bb T^d_N}=\{\eta:\bb T^d_N \to \{0,1\}\}$ so that $\eta(x) =0$
if site $x$ is vacant, and $\eta(x)=1$ if site $x$ is occupied.

The  exclusion process with conductances  is  a continuous-time Markov process
$\{\eta_t : t\ge 0\}$ with state space $\{0,1\}^{\bb T^d_N} $,
whose generator $L_N$ acts on functions $f:
\{0,1\}^{\bb T^d_N} \to \bb R$ as
\begin{equation}
\label{g4}
(L_N f) (\eta) \;=\;\sum^d_{j=1} \sum_{x \in \bb T^d_N} \xi_{x,x+e_j}c_{x,x+e_j}(\eta)\,
\{ f(\sigma^{x,x+e_j} \eta) - f(\eta) \} \;,
\end{equation}
where $\sigma^{x,x+e_j} \eta$ is the configuration obtained from $\eta$
by exchanging the variables $\eta(x)$ and $\eta(x+e_j)$:
\begin{equation}
\label{g5}
(\sigma^{x,x+e_j} \eta)(y) \;=\;
\begin{cases}
\eta (x+e_j) & \text{ if } y=x,\\
\eta (x) & \text{ if } y=x+e_j,\\
\eta (y) & \text{ otherwise}.
\end{cases}
\end{equation}

We consider the Markov process  $\{\eta_t : t\ge 0\}$ on the configurations $\{0,1\}^{\bb T^d_N}$
associated to the generator $L_N$ in the diffusive scale, i.e., $L_N$ is speeded up by $N^2$.

We now describe the stochastic evolution of the process. Let
$x=(x_1,\ldots,x_d)\in \bb T^d_N$. After a time given by an exponential distribution, at rate $\xi_{x,x+e_j}c_{x,x+e_j}(\eta)$ the occupation variables
$\eta(x)$ and $\eta(x+e_j)$ are exchanged. 
Note that only nearest neighbor jumps are allowed.
If $W$ is differentiable at $x/N\in[0,1)^d$, the rate at
which particles are exchanged is of order $1$ for each direction, but if some
$W_j$ is discontinuous at $x_j/N$ , it no longer holds.
In fact, assume, to fix ideas, that $W_j$ is discontinuous at $x_j/N$, and smooth on the segments
$(x_j/N, x_j/N + \varepsilon e_j)$ and $(x_j/N - \varepsilon e_j, x_j/N )$. Assume, also, that $W_k$ is
differentiable in a neighborhood of $x_k/N$ for $k \neq j$. In this case, the
rate at which particles jump over the bonds $\{y-e_j, y\}$, with $y_j=x_j$, is of order
$1/N$, whereas in a neighborhood of size $N$ of these bonds, particles
jump at rate $1$. Thus, note that a particle at site $y-e_j$ jumps to $y$
at rate $1/N$ and jumps at rate $1$ to each one of the $2d-1$ other options.
Particles, therefore, tend to avoid the bonds $\{y-e_j,y\}$. However,
since time will be scaled diffusively, and since on a time interval of length $N^2$ a
particle spends a time of order $N$ at each site $y$, particles will be able to cross the slower
bond $\{y - e_j, y\}$.
The scaling limits of this interacting particle systems in inhomogeneous media may, for instance,  model diffusions in which permeable membranes, at the points of discontinuities of the conductances $W$, tend to reflect particles, creating space discontinuities in the density profiles. For more details  see \cite{v}.

The effect of the factor $c_{x,x+e_j}(\eta)$ is the following: if the parameter $b$ is positive, the presence of particles in the neighboring
sites of the bond $\{x,x+e_j\}$ speeds up the exchange rate by a factor of
order one, and if the parameter $b$ is negative, the presence of particles in the neighboring sites slows down the exchange rate also by a factor of order one. More details are given in Remark \ref{taxac} below.

Let $A =(a_{jj})_{d\times d}$ be a diagonal matrix belonging to $\mathbb{M}_W$ with $a_{jj}>0, j=1,\ldots,d$, and recall from subsection \ref{ns}, the operator defined on $\mf D_W(\bb T^d)$:
\begin{equation*}
\nabla A\nabla_W := \sum_{j=1}^d \partial_{x_j}a_{jj} \partial_{W_j},
\end{equation*}
where $A\in\mathbb{M}_W$.

A sequence of probability measures $\{\mu_N : N\geq 1 \}$ on $\{0,1\}^{\bb T^d_N}$
is said to be associated to a profile $\rho_0 :\bb T^d \to [0,1]$ if for every $\delta>0$ and every function $H\in \mf D_W(\bb T^d)$:
\begin{equation}
\label{f09}
\lim_{N\to\infty}
\mu_N \left\{ \, \eta;\;\Big\vert \frac 1{N^d} \sum_{x\in\bb T^d_N} H(x/N) \eta(x)
- \int H(u) \rho_0(u) du \Big\vert > \delta \right\} \;=\; 0.
\end{equation}

Let $\gamma : \bb T^d \to [l,r]$ be a bounded density profile and consider the parabolic differential equation
\begin{equation}
\label{g03}
\left\{
\begin{array}{l}
{\displaystyle \partial_t \rho \; =\; \nabla A\nabla_W \Phi(\rho) } \\
{\displaystyle \rho(0,\cdot) \;=\; \gamma(\cdot)}
\end{array}
\right. ,
\end{equation}
where the function $\Phi:[l,r]\to\bb R$ has bounded derivative, its derivative is also away from zero, and $t\in[0,T]$, for $T>0$ fixed.

A function $\rho : \bb [0,T] \times \bb T^d \to [l,r]$
is said to be a weak solution of the parabolic differential equation \eqref{g03} if the following conditions hold.  $\Phi(\rho(\cdot,\cdot))$ and $\rho(\cdot,\cdot)$ belong to $L^2([0,T],H_{1,W}(\bb T^d))$, and we have the integral identity
\begin{eqnarray*}
\int_{\bb T^d} \rho(t,u) H(u)du - \int_{\bb T^d} \rho(0,u) H(u)du=
\int_0^t \, \int_{\bb T^d} \Phi (\rho(s,u))  \nabla A\nabla_W H(u)du\,ds  \;,\end{eqnarray*}
for every function $H\in \mf D_W(\bb T^d)$ and all $t\in[0,T]$.

Existence of such weak solutions follow from the tightness of the process proved in subsection \ref{ss1}, and from the energy estimate given in \cite[Lemma 6.2]{v}. Uniquenesses of weak solutions was proved in \cite{SV}.
 
 The main result of this Section is the following.
\begin{theorem}
\label{t02}
Fix a continuous initial profile $\rho_0 : \bb T^d \to [0,1]$ and
consider a sequence of probability measures $\mu_N$ on $\{0,1\}^{\bb
  T^d_N}$ associated to $\rho_0$, in the sense of \eqref{f09}. Then, for any $t\ge 0$,
\begin{equation*}
\lim_{N\to\infty}
\bb P_{\mu_N} \left\{ \, \Big\vert \frac 1{N^d} \sum_{x\in\bb T^d_N}
H(x/N) \eta_t(x) - \int H(u) \rho(t,u)\, du \Big\vert
> \delta \right\} \;=\; 0
\end{equation*}
for every $\delta>0$ and every function $H\in \mf D_W(\bb T^d)$. Here, $\rho$
is the unique weak solution of the non-linear equation \eqref{g03}
with $l=0$, $r=1$, $\gamma = \rho_0$ and $\Phi(\alpha) = \alpha + a
\alpha^2$.
\end{theorem}

\begin{remark}\label{taxac}
The specific form of the rates $c_{x,x+e_i}$ is not important, but two
conditions must be fulfilled. The rates must be strictly positive,
they may not depend on the occupation variables $\eta(x)$, $\eta(x+e_i)$,
but they have to be chosen in such a way that the resulting process is
\emph{gradient} (cf. Chapter 7 in \cite{kl} for the definition of
gradient processes).

We may define rates $c_{x,x+e_i}$ to obtain any polynomial $\Phi$ of the
form $\Phi(\alpha) = \alpha + \sum_{2\le j\le m} a_j \alpha^j$, $m\ge
1$, with $1+ \sum_{2\le j\le m} j a_j >0$. Let, for instance, $m=3$.
Then the rates
%\begin{align*}
$$\hat c_{x,x+e_i} (\eta)\;\; =\;\; c_{x,x+e_i} (\eta)\;\;+
b\left\{ \eta(x-2e_i) \eta(x-e_i) + \eta(x-e_i) \eta(x+2e_i) + \eta(x+2e_i)
\eta(x+3e_i)\right\},$$
%\end{align*}
satisfy the above three conditions,
where $c_{x,x+e_i}$ is the rate defined at the beginning of Section 2 and
$a$, $b$ are such that $1+2a + 3b>0$. An elementary computation shows
that
$\Phi(\alpha) = 1 + a \alpha^2 + b \alpha^3$.
\end{remark}
\subsection{Proof of Theorem \ref{t02}}

 A simple computation shows that the Bernoulli product measures
$\{\nu^N_\alpha : 0\le \alpha \le 1\}$ are invariant, in fact
reversible, for the dynamics. The measure $\nu^N_\alpha$ is obtained
by placing a particle at each site, independently from the other
sites, with probability $\alpha$. Thus, $\nu^N_\alpha$ is a product
measure over $\{0,1\}^{\bb T^d_N}$ with marginals given by
\begin{equation*}
\nu^N_\alpha \{\eta : \eta(x) =1\} \;=\; \alpha
\end{equation*}
for $x$ in $\bb T^d_N$. For more details see \cite[chapter 2]{kl}.\\

Consider the random walk $\{X_t\}_{t\ge0}$ of a particle in $\bb T^d_N$ induced by the generator $L_N$ given as follows.
Let $\xi_{x,x+e_j}$ be given by \eqref{rate}. If the particle is on a site $x\in \bb T^d_N$, it will jump to $x+e_j$ with rate $N^2\xi_{x,x+e_j}$. Furthermore, only nearest neighbor jumps are allowed.
The generator $\bb L_N$ of the random walk $\{X_t\}_{t\ge0}$ acts on functions $f:\bb T^d_N \to \bb R$ as
\begin{equation*}
\bb L_N f\left(\frac{x}{N}\right) \; =\; \sum^d_{j=1} \bb L_N^j f\left(\frac{x}{N}\right),
\end{equation*}
where,
\begin{equation*}
\bb L_N^j f\Big(\frac{x}{N}\Big) = N^2\Big\{\xi_{x,x+e_j} \Big[f\Big(\frac{x+e_j}{N}\Big) - f\Big(\frac{x}{N}\Big)\Big]  +
\xi_{x-e_j,x} \Big[f\Big(\frac{x-e_j}{N}\Big) - f\Big(\frac{x}{N}\Big)\Big]\Big\}
\end{equation*}
It is not difficult to see that the following equality holds:
\begin{equation}
\label{opdisc}
\bb L_N f(x/N) = \sum^d_{j=1}\partial^N_{x_j}(a^N_j\partial^N_{W_j}f)(x)\;:=\;\nabla^NA^N\nabla^N_Wf(x).
\end{equation}

 The counting measure $m_N$ on $N^{-1} \bb T^d_N$ is
reversible for this process.
This random walk plays an important role in the proof of the hydrodynamic limit of the process $\eta_t$.

Let $D(\bb R_+, \{0,1\}^{\bb T^d_N})$ be the path space of
c\`adl\`ag trajectories with values in $\{0,1\}^{\bb T^d_N}$. For a
measure $\mu_N$ on $\{0,1\}^{\bb T^d_N}$, denote by $\bb P_{\mu_N}$ the
probability measure on $D(\bb R_+, \{0,1\}^{\bb T^d_N})$ induced by the
initial state $\mu_N$ and the Markov process $\{\eta_t : t\ge 0\}$.
Expectation with respect to $\bb P_{\mu_N}$ is denoted by $\bb
E_{\mu_N}$.

Let $\mc M$ be the space of positive measures on $\bb T^d$ with total
mass bounded by one endowed with the weak topology. Recall that
$\pi^{N}_{t} \in \mc M$ stands for the empirical measure at time $t$.
This is the measure on $\bb T^d$ obtained by rescaling space by $N$ and
by assigning mass $1/N^d$ to each particle:
\begin{equation}
\label{f01}
\pi^{N}_{t} \;=\; \frac{1}{N^d} \sum _{x\in \bb T^d_N} \eta_t (x)\,
\delta_{x/N}\;,
\end{equation}
where $\delta_u$ is the Dirac measure concentrated on $u$.

For a function $H:\bb T^d \to \bb R$, $\<\pi^N_t, H\>$ stands for
the integral of $H$ with respect to $\pi^N_t$:
\begin{equation*}
\<\pi^N_t, H\> \;=\; \frac 1{N^d} \sum_{x\in\bb T^d_N}
H (x/N) \eta_t(x)\;.
\end{equation*}
This notation is not to be mistaken with the inner product in
$L^2(\bb T^d)$ introduced earlier. Also, when $\pi_t$ has a density
$\rho$, $\pi(t,du) = \rho(t,u) du$.

Fix $T>0$ and let $D([0,T], \mc M)$ be the space of $\mc M$-valued
c\`adl\`ag trajectories $\pi:[0,T]\to\mc M$ endowed with the
\emph{uniform} topology. Note that $\mc M$ is endowed with the weak topology, which is metrizable, since $\bb T^d$ is a compact metric space. Thus, this uniform topology is well-defined.  For each probability measure $\mu_N$ on
$\{0,1\}^{\bb T^d_N}$, denote by $\bb Q_{\mu_N}^{W,N}$ the measure on
the path space $D([0,T], \mc M)$ induced by the measure $\mu_N$ and
the process $\pi^N_t$ introduced in \eqref{f01}.

Fix a continuous profile $\rho_0 : \bb T^d \to [0,1]$ and consider a
sequence $\{\mu_N : N\ge 1\}$ of measures on $\{0,1\}^{\bb T^d_N}$
associated to $\rho_0$ in the sense \eqref{f09}. Further, we denote by $\bb Q_{W}$
the probability measure on $D([0,T], \mc M)$ concentrated on the
deterministic path $\pi(t,du) = \rho (t,u)du$, where $\rho$ is the
unique weak solution of \eqref{g03} with $\gamma = \rho_0$, $l_k=0$,
$r_k=1$, $k=1,\ldots,d$ and $\Phi(\alpha) = \alpha + b\alpha^2$.

In subsection \ref{ss1} we show that the sequence $\{\bb Q_{\mu_N}^{W,N} : N\ge
1\}$ is tight, and in subsection \ref{ss2} we characterize the limit
points of this sequence.

\subsection{Tightness}
\label{ss1}
The goal of this subsection is to prove tightness of sequence $\{\bb Q_{\mu_N}^{W,N} : N\ge 1\}$.
Fix $\lambda >0$ and consider, initially, the auxiliary $\mc
M$-valued Markov process $\{\Pi^{\lambda,N}_t : t\ge 0\}$
defined by
\begin{equation*}
\Pi^{\lambda,N}_t (H) \;=\; \< \pi^N_t,H_\lambda^N\>\;=\;
\frac{1}{N^d} \sum _{x\in \bb Z^d} H_\lambda^N(x/N)
\eta_t (x),
\end{equation*}
for $H$ in $\mf D_W(\bb T^d)$, where $H_\lambda^N$ is the unique weak solution in $H_{1,W}(\bb T_N^d)$  of
$$\lambda H_\lambda^N - \nabla^N A^N\nabla_W^N H_\lambda^N = \lambda H - \nabla A \nabla_W H,$$
with the right-hand side being understood as the restriction of the function to the lattice $\bb T_N^d$, which is well-defined, since $H\in\mf D_W(\bb T^d)$
and from Remark \ref{suave}, we have that $\nabla A\nabla_W H$ belongs to $C_W(\bb T^d)$. Thus the right-hand side belongs to $C_W(\bb T^d)$ and it is right-continuous,  Remark \ref{DWrightcontinuous}.

We first prove tightness of the process $\{\Pi^{\lambda,N}_t : 0\le t
\le T\}$. Then we show that $\{\Pi^{\lambda,N}_t
: 0\le t \le T\}$ and $\{\pi^{N}_t : 0\le t \le T\}$ are not far apart.

It is well known \cite{kl} that to prove
tightness of $\{\Pi^{\lambda,N}_t : 0\le t \le T\}$ it is enough to
show tightness of the real-valued processes $\{\Pi^{\lambda,N}_t (H) :
0\le t \le T\}$ for a set of test functions $H:\bb T^d\to \bb R$ dense
in $C(\bb T^d)$ for the uniform topology, for instance we can use $\mf D_W(\bb T^d)$.

Fix a function $H: \bb T^d \to \bb R$ in $\mf D_W(\bb T^d)$. Keep in mind that $\Pi^{\lambda,N}_t (H) = \<\pi^N_t, H_\lambda^N \>$,
and denote by $M^{N,\lambda}_t$ the martingale defined by
\begin{equation}
\label{f10}
M^{N,\lambda}_t \;=\;  \Pi^{\lambda,N}_t (H) \;-\;
\Pi^{\lambda,N}_0 (H) \;-\; \int_0^t ds \, N^2 L_N \<\pi^N_s ,
H_\lambda^N \> \;.
\end{equation}
Clearly, tightness of $\Pi^{\lambda,N}_t (H)$ follows from tightness
of the martingale $M^{N,\lambda}_t$ and tightness of the additive
functional $\int_0^t ds \, N^2 L_N \<\pi^N_s , H_\lambda^N \>$.

A long computation, albeit simple, shows that the quadratic variation
$\<M^{N,\lambda}\>_t$ of the martingale $M^{N,\lambda}_t$ is given by:
%\begin{align*}
$$ \frac{1}{N^{2d-1}}\sum^d_{j=1}\sum_{x\in \bb T^d}[\partial_{W_j}^N H^N_{\lambda}(x/N)]^2[W((x+e_j)/N) - W(x/N)]\times\int_0^t c_{x,x+e_j}(\eta_s) \, [\eta_s(x+e_j) - \eta_s(x)]^2 \, ds\;.$$
%\end{align*}
 In particular, by Lemma \ref{lmdiscreto},
\begin{equation*}
  \<M^{N,\lambda}\>_t \;\le\; \frac{C_0 t}{N^{d}} \sum^d_{j=1} \|H_\lambda^N\|_{W_j,N}^2 \;\le\; \frac{C(H)t}{\lambda N^d},
\end{equation*}
for some finite constant $C(H)$, which depends only on $H$. Thus, by
Doob inequality, for every $\lambda>0$, $\delta>0$,
\begin{equation}
\label{f02}
\lim_{N\to\infty} \bb P_{\mu_N} \left[ \sup_{0\le t\le T}
\big\vert M^{N,\lambda}_t \big\vert \, > \, \delta \right]
\;=\; 0\;.
\end{equation}
In particular, the sequence of martingales $\{M^{N,\lambda}_t : N\ge
1\}$ is tight for the uniform topology.

It remains to examine the additive functional of the decomposition
\eqref{f10}. The  generator of the exclusion process $L_N$  can be decomposed
 in terms of the generator of the random walk $\bb L_{N}$.  By a long but simple computation, we obtain that
 $N^2 L_N \<\pi^N , H_\lambda^N \>$ is equal to
\begin{eqnarray*}
\!\!\!\!\!\!\!\!\!\!\!\!\!\! &&
\sum^d_{j=1}\big \{\frac {1}{N^d} \sum_{x\in \bb T^d_N} (\bb L_N^j H_\lambda^N)(x/N)\, \eta(x)- \; \frac{b}{N^d} \sum_{x\in \bb T^d_N} (\bb L_N^j H_\lambda^N)
(x/N) (\tau_x h_{2,j}) (\eta)\big \}\;
\\
\!\!\!\!\!\!\!\!\!\!\!\!\!\! && \quad
+\; \frac{b}{N^d} \sum_{x\in \bb T^d_N} \big [ (\bb L_N^j H_\lambda^N)
((x+e_j)/N) + (\bb L_{N}^j H_\lambda^N) (x/N) \big ] \,
(\tau_x h_{1,j}) (\eta), \\
\end{eqnarray*}
where $\{\tau_x: x\in \bb Z^d\}$ is the group of translations, so that
$(\tau_x \eta)(y) = \eta(x+y)$ for $x$, $y$ in $\bb Z^d$, and the
sum is understood modulo $N$. Also, $h_{1,j}$, $h_{2,j}$ are the cylinder functions
\begin{equation*}
h_{1,j}(\eta) \;=\; \eta(0) \eta({e_j})\;,\quad h_{2,j}(\eta) \;=\; \eta(-e_j)  \eta(e_j)\;.
\end{equation*}

Since $H_\lambda^N$ is the weak solution of the discrete equation, we have by Remark \ref{diracdiscreta} that it is also a strong solution. Then, we may replace $\bb L_N H_\lambda^N$ by $U_\lambda^N
= \lambda (H_\lambda^N - H)+\nabla A\nabla_W H$ in the previous formula. In particular,
for all $0\le s<t\le T$,
\begin{equation*}
\Big\vert \int_s^t dr \, N^2 L_N \<\pi^N_r ,H_\lambda^N \> \Big\vert
\;\le\; \frac {(1+3|b|)(t-s)}{N^d} \sum_{x\in \bb T^d_N} |U_\lambda^N (x/N)|
\;.
\end{equation*}
It follows from the estimate given in Lemma \ref{lmdiscreto}, Remark \ref{suave}, Lemma \ref{right-cont-bounded}, and from Schwartz
inequality, that the right hand side of the previous expression is bounded above by $C(H,b) (t-s)$
uniformly in $N$, where $C(H,b)$ is a finite constant depending only
on $b$ and $H$. This proves that the additive part of the
decomposition \eqref{f10} is tight for the uniform topology and,
therefore, that the sequence of processes $\{\Pi^{\lambda,N}_t :N\ge
1\}$ is tight.

\begin{lemma}
\label{s06}
The sequence of measures $\{\bb Q_{\mu^N}^{W,N} : N\ge 1\}$ is tight
for the uniform topology.
\end{lemma}

\begin{proof}
Fix $\lambda > 0$. It is enough to show that for every function $H\in \mf D_W(\bb T^d)$
and every $\epsilon>0$, we have
\begin{equation*}
\lim_{N\to\infty} \bb P_{\mu^N} \left[
\sup_{0\le t\le T} |\, \Pi^{\lambda,N}_t (H) -
\<\pi^N_t, H\>\, | > \epsilon
\right] \;=\;0,
\end{equation*}
whence tightness of $\pi^N_t$ follows from
tightness of $\Pi^{\lambda,N}_t$. By Chebyshev's inequality, the last expression is bounded above by
$$\frac{1}{\epsilon^2}\bb E_{\mu_N} \left[\sup_{0\le t\le T} |\, \Pi^{\lambda,N}_t (H) -
\<\pi^N_t, H\>\, |^2\right] \leq \frac{2}{\epsilon^2}\|H_\lambda^N - H\|_{L^2(\bb T_N^d)}^2,
$$
since there exists at most one particle per site. By Theorem \ref{homoge}, Proposition \ref{hconvergencia} and Corollary \ref{suave-hconv}, $\|H_\lambda^N - H\|_{L^2(\bb T_N^d)}^2\to 0$ as $N\to\infty$, and the proof follows.
\end{proof}

The reader should compare the proof given in this Section to the proof given in \cite{v}, to check that the one given here follows closely
the standard approach given, for instance, in \cite{kl}, whereas the approach given in \cite{v} is non-standard. This means that the theory
provided in this article made the study of hydrodynamic behavior of exclusion processes with conductances tractable in the standard fashion.
Thus, future work on this field will be simplified.

\subsection{Uniqueness of limit points}
\label{ss2}

We prove in this subsection that all limit points $\bb Q^*$ of the
sequence $\bb Q^{W,N}_{\mu_N}$ are concentrated on absolutely
continuous trajectories $\pi(t,du) = \rho(t,u) du$, whose density
$\rho(t,u)$ is a weak solution of the hydrodynamic equation
\eqref{g03} with $l=0$, $r=1$ and $\Phi(\alpha)=\alpha + a\alpha^2$.

We now state a result necessary to prove the uniqueness of limit points. Let, for a local function $g: \{0,1\}^{\bb Z^d} \to \bb R$,
$\widetilde g :[0,1]\to \bb R$ be the expected value of $g$ under the
stationary states:
\begin{equation*}
\widetilde g (\alpha) \;=\; E_{\nu_\alpha} [ g(\eta)]\;.
\end{equation*}
For $\ell \ge 1$ and $d$-dimensional integer $x=(x_1,\ldots,x_d)$, denote
by $\eta^{\ell} (x)$ the empirical density of particles in the box
$ \bb B_+^\ell(x)= \{(y_1,\ldots,y_d)\in\bb Z^d\;;0\le y_i-x_i < \ell\}$:
\begin{equation*}
  \eta^{\ell} (x) \;=\; \frac{1}{\ell^d}  \sum_{y\in \bb B_+^\ell(x)} \eta(y)\;.
\end{equation*}

Let $\bb Q^*$ be a limit point of the sequence $\bb Q^{W,N}_{\mu_N}$
and assume, without loss of generality, that $\bb Q^{W,N}_{\mu_N}$
converges to $\bb Q^*$.

Since there is at most one particle per site, it is clear that $\bb
Q^*$ is concentrated on trajectories $\pi_t(du)$ which are absolutely
continuous with respect to the Lebesgue measure, $\pi_t(du) =
\rho(t,u) du$, and whose density $\rho$ is non-negative and bounded by
$1$. The reader is referred to \cite[Chapter 4]{kl} for further details.

Fix a function $H\in\mf D_W(\bb T^d)$  and
$\lambda>0$.  Recall the definition of the martingale
$M^{N,\lambda}_t$ introduced in the previous Section. From \eqref{f02} we have,
for every $\delta>0$,
\begin{equation*}
\lim_{N\to\infty} \bb P_{\mu_N} \left[ \sup_{0\le t\le T}
\big\vert M^{N,\lambda}_t \big\vert \, > \, \delta \right]
\;=\; 0\;,
\end{equation*}
and from \eqref{f10}, for fixed $0<t\le T$ and $\delta>0$, we have
\begin{equation*}
\lim_{N\to\infty} \bb Q^{W,N}_{\mu_N} \left[ \,
\Big\vert \<\pi^N_t, H^N_\lambda \> \;-\;
\<\pi^N_0, H^N_\lambda \> \;-\;
\int_0^t ds \, N^2 L_N \<\pi^N_s , H^N_\lambda \>
\Big\vert \, > \, \delta \right] \;=\; 0.
\end{equation*}

Note that the expression $N^2 L_N \<\pi^N_s , H^N_\lambda\>$ has been computed in the previous subsection in terms of generator $\bb L_N$. On the other hand, $\bb L_N H_\lambda^N = \lambda H_\lambda^N - \lambda H + \nabla A\nabla_W H$. Since there is at most one particle per site, we may apply Theorem \ref{homoge} along with Proposition \ref{hconvergencia} and Corollary \ref{suave-hconv}, to replace $\<\pi^N_t,
H^N_\lambda \>$ and $\<\pi^N_0, H^N_\lambda \>$ by $\<\pi_t, H\>$ and $\<\pi_0,H\>$, respectively, and replace $\bb L_N H_\lambda^N$ by $\nabla A\nabla_W H$ plus a term that vanishes as $N\to\infty$.

Since $E_{\nu_\alpha}[h_{i,j}] = \alpha^2$, $i=1$, $2$ and $j = 1,\ldots, d$, we have by the replacement Lemma \cite[Corollary 5.4]{v} that, for every $t>0$, $\lambda>0$,
$\delta>0$, $i=1$, $2$,
%\begin{align*}
$$\lim_{\varepsilon \to 0} \limsup_{N\to\infty}
\bb P_{\mu_N} \Big[ \, \Big| \int_0^t \!\!\! ds\, \frac 1{N^d}
\sum_{x\in \bb T^d_N} \nabla A\nabla_W H  (x/N)
\times\left\{ \tau_x h_{i,j} (\eta_s) -
\left[\eta^{\varepsilon N}_s(x)\right]^2 \right\} \, \Big|
\, > \, \delta \, \Big]  \;=\; 0.$$
%\end{align*}

Since $\eta^{\varepsilon N}_s(x) = \varepsilon^{-d} \pi^N_s (\prod_{j=1}^d[x_j/N, x_j/N + \varepsilon e_j])$,
 we obtain, from the previous considerations, that
%\begin{align*}
$$\lim_{\varepsilon \to 0} \limsup_{N\to\infty} \bb Q^{W,N}_{\mu_N} \left[ \,
\Big\vert\right. \<\pi_t, H \>
-\; \<\pi_0, H \> \;-\;\left.
\int_0^t ds \, \Big\< \Phi \big(\varepsilon^{-d} \pi^N_s (\prod_{j=1}^d[\cdot, \cdot
+ \varepsilon e_j]) \big) \,,\, \nabla A\nabla_W H\Big>
\Big\vert > \delta \right] \;=\; 0\;.$$
%%\end{align*}

Using the fact that $\bb Q^{W,N}_{\mu_N}$ converges in the uniform topology to $\bb
Q^*$, we have that
%\begin{align*}
$$\lim_{\varepsilon \to 0}\bb Q^*\left[ \,
\Big\vert \<\pi_t, G_\lambda H \> \right.\;-\; \; \<\pi_0, G_\lambda H \>
-\;\int_0^t ds \, \left.\Big\< \Phi \big (\varepsilon^{-d} \pi_s (\prod_{j=1}^d[\cdot, \cdot
+ \varepsilon e_j]) \big) \,,\, U_\lambda\Big>
\Big\vert > \delta \right] \;=\; 0\;.$$
%\end{align*}

Recall that $\bb Q^{*}$ is concentrated on absolutely continuous paths
$\pi_t(du) = \rho(t,u) du$ with positive density bounded by $1$. Therefore,
$\varepsilon^{-d}\pi_s(\prod_{j=1}^d[\cdot, \cdot + \varepsilon e_j])$ converges in
$L^1(\bb T^d)$ to $\rho(s,.)$ as $\varepsilon\downarrow 0$. Thus,
\begin{eqnarray*}
\bb Q^{*} \left[ \,
\Big\vert \<\pi_t, H \> \;-\;
 \<\pi_0, H \> \;-\;
\int_0^t ds \, \< \Phi (\rho_s) \,,\, \nabla A\nabla_W H \>
\Big\vert > \delta \right] \;=\; 0.
\end{eqnarray*}
Letting $\delta\downarrow
0$, we see that, $\bb Q^{*}$ a.s.,
\begin{eqnarray*}
\int_{\bb T^d} \rho(t,u) H(u)du - \int_{\bb T^d} \rho(0,u) H(u)du=
\int_0^t \, \int_{\bb T^d} \Phi (\rho(s,u))  \nabla A\nabla_W H(u)du\,ds  \;.
\end{eqnarray*}

This identity can be extended to a countable set of times $t$. Taking
this set to be dense we
obtain, by continuity of the trajectories $\pi_t$, that it holds for all $0\le t\le T$.

From \cite[Lemma 6.2]{v}, which we may easily adapt to our setup by using the uniform ellipticity condition \eqref{unif-elip} of the environment, we may conclude that all limit points have, almost surely, finite energy, and therefore, by \cite[Lemma 4.1]{SV}, $\Phi(\rho(\cdot,\cdot))\in L^2([0,T],H_{1,W}(\bb T^d))$. Analogously, it is possible to show that
 $\rho(\cdot,\cdot)$ has finite energy and hence it belongs to $L^2([0,T],H_{1,W}(\bb T^d))$.
\begin{proposition}
\label{s15}
As $N\uparrow\infty$, the sequence of probability measures $\bb
Q_{\mu_N}^{W,N}$ converges in the uniform topology to $\bb Q_{W}$.
\end{proposition}
\begin{proof}
In the previous subsection, we showed that the sequence of probability
measures $\bb Q^{W,N}_{\mu_N}$ is tight for the uniform topology. Moreover, we
just proved that all limit points of this sequence are concentrated on
weak solutions of the parabolic equation \eqref{g03}. The proposition now follows
 from the uniqueness proved in \cite{SV}.
\end{proof}

\begin{proof}[Proof of Theorem \ref{t02}]
Since $\bb Q_{\mu_N}^{W,N}$ converges in the uniform topology to $\bb
Q_{W}$, a measure which is concentrated on a deterministic path, for
each $0\le t\le T$ and each continuous function $H:\bb T^d\to \bb R$,
$\<\pi^N_t, H\>$ converges in probability to $\int_{\bb T^d} du
\rho(t,u)H(u)$, where $\rho$ is the unique weak solution of
\eqref{g03} with $l_k=0$, $r_k=1$, $\gamma=\rho_0$ and $\Phi(\alpha) =
\alpha + a \alpha^2$.
\end{proof}

\section{APPENDIX} 
%\label{apendice}

In this appendix we will present, for the reader's convenience, part of the proof of the \cite[ Lemma 1]{TC}. This result was used in Proposition \ref{d1}.
\begin{proposition}\label{prova d1}
The space 
${\mf D}_W(\bb T)$ is dense in $C(\bb T),$ the space of continuous functions in $\bb T$, in the sup norm, $\|\cdot\|_\infty$.
\end{proposition}

\begin{proof}
Let $f:\bb T\to \bb R$ be a continuous function, and $\epsilon
>0$. From uniform continuity, there exists $\delta >0$ such that $|f(y)-f(x)|\le \epsilon$ 
whenever $|x-y|\le \delta$. 
Choose an integer $n \ge \delta^{-1}$ and consider
the function $g:\bb T \to \bb R$ defined by
\begin{equation*}
g(x) \;=\; \sum_{j=0}^{n-1}\frac {f([j+1]/n) - f(j/N)}{W_k([j+1]/n) -
  W_k(j/N)} \mb 1\{(j/n, (j+1)/n]\}(x)\; ,
\end{equation*}
with $\mb 1\{A\}$ being the indicator function of the set $A$. $g$ can be seen 
as a discrete derivative of $f$ with respect to $W_k$. Thus, it is natural
to guess that integrating this function with respect to $W_k$ would yield
an approximation of $f$. Thus, let $G:
\bb T \to \bb R$ be given by $G(x) = f(0) + \int_{(0,x]} g(y)
W_k(dy)$. By the very definition of $g$, $G(j/n) = f(j/n)$ for $0\le j <
n$. 

Since $n\ge \delta^{-1}$, we can use the definition of $G$ to obtain, for $j/n \le
x \le (j+1)/n$,
\begin{equation*}
\big| G(x) - f(x) \big| \;\le\; \big| G(x) - G(j/n) \big| \;+\;
\big| f(x) - f(j/n) \big| \;\le\; 2 \epsilon\;,
\end{equation*}
whence $\Vert G - f \Vert_\infty \le 2 \epsilon$. Note that
\begin{equation}
\label{f15}
\int_{(0,1]} g\, dW_k \;=\;0\;.
\end{equation}

It remains to be shown that the function $G$ may be approximated in
the sup norm by functions in $\mf D_W(\bb T^d)$. Note that the only restriction
we had when choosing the set $\{0, 1/n, \ldots, (n-1)/n\}$ is that the
distance between two consecutive points is less than $\delta$. Therefore,
we may replace these $n$ points, by any such points satisfying this restriction.
Since $W_k$ is strictly increasing, it has a countable number of discontinuities,
and then, we may assume, without loss of generality, that $W_k$ is continuous at
the points $\{0,1/n,\ldots,(n-1)/n\}$ (in the sense, that we may replace these points
if it is not the case). Let $\{H_m : m\ge 1\}$ be a sequence of smooth
functions $H_m : \bb T \to \bb R$, with $|H_m(x)|\leq \Vert
g\Vert_\infty$, for all $x\in \bb T$, and such that $\lim_m H_m(x) = g(x)$ for $xn \not\in
\bb Z$. Then, the Dominated Convergence Theorem implies
\begin{equation}
\label{f16}
\lim_{m\to\infty} \int_{\bb T} \big\vert H_m(y) - g(y) \big|\,
dW_k(y)\;=\; 0\;.
\end{equation}

Let $\{F_m : m\ge 1\}$ be the sequence of functions $F_m:\bb T\to\bb
R$ defined by
$$F_m(x) = f(0) + \int_{(0,x]} H(y) dW_k(y).$$

Therefore, it is immediate that
$$\|G-F_m\|_\infty \le \int_{\bb T} |H_m(y)-g(y)| dW_k(y) \stackrel{m\to\infty}{\longrightarrow} 0.$$

This shows that we can approximate, in the sup norm, any function $f\in C(\bb T)$ by functions of the form of $F_m$. To conclude that 
$\mf D_{W_k}$ is dense in $C(\bb T)$, we must show that $F_m\in \mf D_{W_k}$. To this end, note that
\begin{eqnarray*}
F_m(x)&=& f(0) \;+\; \int_{(0,x]} \Big \{ b_m + \int_0^y H'_m(z) \,
dz \Big \} dW_k(y) \\
&=& f(0) \;+\; b_m \, W_k(x) \;+\;
\int_{(0,x]} dW_k(y) \int_0^y H'_m(z) \, dz \;,
\end{eqnarray*}
where $b_m = H_m(0) - W_k(1)^{-1} \int_{\bb T} H_m(y) \, dW_k(y)$. Since $H'_m$ is continuous, $H'_m\in C_W(\bb T)$. 
Then, using the characterization of $\mf D_{W_k}$, one may conclude that $F_m$ belongs to $\mf D_{W_k}$, for each $m\ge 1$.
\end{proof}

\end{document}